%% file: main_multicritical_bounded.tex
\documentclass[11pt,a4paper]{amsart}
\usepackage{amsmath,amsfonts,amssymb,amscd,amsthm, tikz}
\usepackage[shortlabels]{enumitem}

\usepackage{color, soul}
\definecolor{gr}{rgb}{0.7, 1, 0.7}
\definecolor{rr}{rgb}{1, 0.7, 0.7}
\newcommand{\hlc}[2][yellow]{ {\sethlcolor{#1} \hl{#2}} }

\numberwithin{equation}{section}

\usepackage{latexsym}
\usepackage{graphicx}

\theoremstyle{plain} 
\newtheorem{theorem}{Theorem}[section]
\newtheorem{lemma}[theorem]{Lemma}

\newtheorem{proposition}[theorem]{Proposition}

\theoremstyle{definition} 
\newtheorem{definition}[theorem]{Definition}

\theoremstyle{remark} 
\newtheorem{remark}[theorem]{Remark}

\renewcommand{\mathfrak}{\mathbf}

\renewcommand{\Im}{\,\mathrm{Im}\,}

\newcommand{\ignore}[1]{}

\newcommand{\eps}{\varepsilon}

\newcommand{\cA}{\mathcal{A}}

\newcommand{\bbC}{\mathbb{C}}
\newcommand{\bbH}{\mathbb{H}}

\newcommand{\bbN}{\mathbb{N}}
\newcommand{\bbR}{\mathbb{R}}
\newcommand{\bbZ}{\mathbb{Z}}
\newcommand{\bbQ}{\mathbb{Q}}
\newcommand{\bbT}{\mathbb{T}}

\newcommand{\tl}{\tilde}
\newcommand{\QQ}{\mathbb Q}
\newcommand{\NN}{\mathbb N}
\newcommand{\CC}{\mathbb C}
\newcommand{\ZZ}{\mathbb Z}

\newcommand{\RR}{\mathbb R}
\newcommand{\TT}{\mathbb T}
\newcommand{\cR}{\mathcal R}
\newcommand{\cI}{\mathcal I}

\newcommand{\cC}{\mathcal C}

\newcommand{\cH}{\mathcal H}

\newcommand{\dist}{\operatorname{dist}}
\newcommand{\diam}{\operatorname{diam}}

\renewcommand{\Im}{\operatorname{Im}}
\renewcommand{\mod}{\operatorname{mod}}

\begin{document}
\title[Rigidity of multicritical circle maps]{Rigidity of analytic and smooth bi-cubic multicritical circle maps with bounded type rotation numbers}

\begin{abstract}
We prove that if two analytic multicritical circle maps with the same bounded type rotation number are topologically conjugate by a conjugacy which matches the  critical points of the two maps while preserving the orders of their criticalities, then the conjugacy necessarily has $C^{1+\alpha}$ regularity, where $\alpha$ depends only on the bound on the type of the rotation number. 
We then extend this rigidity result to $C^3$-smooth bi-cubic circle maps.
\end{abstract}

\author{Igors Gorbovickis}
\address{Jacobs University, Bremen, Germany}
\email{i.gorbovickis@jacobs-university.de}

\author{Michael Yampolsky}
\address{University of Toronto, Toronto, Canada}
\email{yampol@math.toronto.edu}

\thanks{I.G. was partially supported by the Deutsche Forschungsgemeinschaft (DFG, German Research Foundation) – project number 455038303. M.Y. was partially supported by NSERC Discovery Grant.}

\subjclass[2010]{}

\date{\today}
\maketitle

\input{preliminaries-bounded}

\input{holpairs}

\input{proof}

\newpage

\bibliographystyle{amsalpha}
\bibliography{biblio}
\end{document}

%% file: preliminaries-bounded.tex
\section{Introduction}

\subsection{Rigidity and renormalization of circle maps with critical points}
The study of rigidity of circle homeomorphisms is one of the most classical and most central questions of low-dimensional dynamics, which ties together such subjects as the KAM theory and Renormalization theory. Its history can be traced to the celebrated early result by Denjoy~\cite{Denjoy} who showed that if $f$ and $g$ are two smooth diffeomorphisms of the circle whose derivatives have bounded variation, and the rotation numbers $\rho(f)=\rho(g)\notin \QQ$, then there exists a homeomorphism of the circle $\phi$ such that
\begin{equation}
  \label{eq:conj}
  \phi\circ f=g\circ\phi.
  \end{equation}
The condition $\rho(f)=\rho(g)\notin \QQ$ is known as {\it combinatorial equivalence}; by the results of Poincar\'e it means that the order in which  the points in every orbit  are arranged  on the circle is the same for the two maps. 
In general, we speak of smooth rigidity, when in some class of circle maps, a weak notion of equivalence, such as combinatorial equivalence, leads to the existence of a conjugacy in the appropriate smoothness class.

Rigidity question has been one of the main motivations for the development of  Renormalization theory of {\it critical circle maps}, that is, smooth homeomorphisms of the circle with a single critical point, originating in \cite{FKS,ORSS}. In the analytic setting, this Renormalization theory was completed by the second author \cite{Ya3,Ya4,Yampolsky_unicritical}, who, in particular, established that for any two analytic $f$ and $g$ with the critical points of the same order and with $\rho(f)=\rho(g)\notin\QQ$, the distance between renormalizations $\cR^nf$, $\cR^n g$ decreases geometrically fast.  Together with D.~Khmelev \cite{KhYam}, he further showed that for two analytic critical circle maps $f$ and $g$ with the same irrational rotation number geometric convergence of renormalizations  is equivalent to the existence of a conjugacy $\phi$ in (\ref{eq:conj}) which is $C^{1+\alpha}$-smooth at the critical point. Previous results about convergence of renormalizations and $C^{1+\alpha}$-smoothness of the conjugacy $\phi$ were obtained by de~Faria and de~Melo in~\cite{FM1,FM2} for rotation numbers of bounded type (irrational numbers of bounded type have bounded terms in their continued fraction expansions, which is equivalent to the Diophantine condition of order $2$).

In the study of  circle homeomorphisms with critical points, there is an important distinction between the case of rotation numbers of bounded type and the general (unbounded type) case. For instance, in the bounded type case,
$C^{1+\alpha}$-smoothness of the conjugacy at the critical point propagates to the whole circle, which is no longer true in the unbounded type case. Moreover, in the case of rotation numbers of bounded type, Guarino and de~Melo \cite{GdM} extended the $C^{1+\alpha}$-rigidity result to critical circle maps of class $C^3$ with critical points of an odd integer order; the situation is considerably more complex in the unbounded type case (see \cite{GMdM}).

In the present paper we consider the case of {\it multicritical circle maps}, which possess several critical points.
For a survey of some of the previous results on multicritical circle maps, see  \cite{dFGuar}.
For such maps, the equality of the rotation number and of the number and the degeneracy of each of the critical points can not be sufficient for a rigidity statement to hold. It is clearly necessary that critical points with the same criticality, as well as their iterates, appear in the same order on the circle: a condition we discuss below in some detail, and characterize by {\it a signature} of the map (cf,~Definition~\ref{signature_def}).


We settle two rigidity conjectures: one for the general case analytic maps of bounded type, the other for a more narrow class of $C^3$-smooth maps.
We first prove that renormalizations of two analytic multicritical circle maps of bounded type with the same rotation number and the same signature
converge together geometrically fast. This has long been conjectured based on the analogous result for unicritical maps \cite{FM2,KhYam,Ya4}, but has only been shown previously in \cite{Est_Smania_Yampolsky_2020} for maps with two cubic critical points (bi-cubic maps). The argument we give is fundamentally different from that of \cite{Est_Smania_Yampolsky_2020}, as it does not rely on the existence of rigid model families of maps with given criticalities. From this, we derive the $C^{1+\alpha}$-rigidity result for all analytic multicritical maps of bounded type.

We next turn our attention to $C^3$-smooth maps. Here, we develop a new approach to proving rigidity based on hyperbolicity of the renormalization horseshoe in the analytic class. Hyperbolicity was recently shown for bi-cubic maps of bounded type by G.~Estevez and the second author \cite{EsYam2021}. Building on this, we prove that geometric convergence of renormalizations and $C^{1+\alpha}$-rigidity hold for $C^3$-smooth bi-cubic maps of bounded type. We note that our approach to smooth rigidity is quite different from that previously  used for smooth unicritical maps in \cite{GdM}.
It should be applicable to maps with an arbitrary criticality, once the hyperbolicity of the renormalization horseshoe is known.


\subsection{Statements of the main results}
Let us abbreviate a finite or infinite continued fraction with coefficients $r_i\in\NN$, $i\geq 0$ as $[r_0,r_1,\ldots]$.  For an irrational number $\rho\in(0,1)$ there is a unique infinite expansion
$$\rho=[r_0,r_1,\ldots],$$
and the $n$-th continued fraction convergent of $\rho$ is given by
$$p_n/q_n=[r_0,\ldots,r_{n-1}].$$
A number $\rho\in(0,1)\setminus\QQ$ is of \emph{bounded type} if the coefficients $r_i$ in its expansion are bounded. We will say that $\rho$ is of a {\it type  bounded by} $B>0$ if
$$\sup r_i<B.$$
This is, of course, equivalent to the second order Diophantine condition:
\begin{equation}\label{bounded_type_def}
\left|\rho-\frac{p}{q}\right|\geq\frac{M}{q^2}\quad\mbox{for any integers $p$ and $q \neq 0$}
\end{equation}
for some $M=M(B)$.

In the theorems below, the total criticality of a map refers to the product of the criticalities of its critical points. The latter ones are always assumed to be odd integers.
Here are our principal renormalization results:

\begin{theorem}[{\bf Renormalization convergence for analytic multicritical circle maps}]
  \label{th:rencov2}
For any positive integers $B>0$ and $D\ge 3$, there exists a constant $\lambda=\lambda(B, D)\in (0,1)$ that satisfies the following property: given two analytic multicritical circle maps $f$ and $g$ with the same signature, the same irrational rotation number $\rho$ of combinatorial type bounded by $B$ and the total criticality $D(f)=D(g)$ not greater than $D$, there exists $C>0$ such that for all $n\in\bbN$,
$$
\dist_{C^0}\big(\mathcal{R}^n(f),\mathcal{R}^n(g)\big) \leq C\lambda^n.
$$
\end{theorem}

\noindent
From this, we derive:
\begin{theorem}[{\bf Rigidity for analytic multicritical circle maps}]
  \label{thm-rigidity-analytic}
  For any two 
  analytic multicritical circle maps $f$ and $g$ with the same irrational rotation number $\rho$ of bounded type and the same signature,
  there exists a conjugacy $\phi$ which is $C^{1+\alpha}$-smooth, where $\alpha>0$ depends on $M$ from~(\ref{bounded_type_def}) and on the total criticality $D(f)$.
\end{theorem}

\noindent
In the $C^3$-smooth category, we show:
\begin{theorem}[{\bf Renormalization convergence for $C^3$ bi-cubic maps}]
  \label{th:rencov1}
For any positive integer $B>0$, there exists a constant $\lambda=\lambda(B)\in (0,1)$ that satisfies the following property: given two $C^3$ bi-cubic circle maps $f$ and $g$ with the same signature and the same irrational rotation number $\rho$ of combinatorial type bounded by $B$, there exists $C>0$ such that for all $n\in\bbN$,
$$
\dist_{C^0}\big(\mathcal{R}^n(f),\mathcal{R}^n(g)\big) \leq C\lambda^n.
$$

\end{theorem}

\noindent
As a consequence of Theorem~\ref{th:rencov1} we obtain:
\begin{theorem}[{\bf Rigidity for $C^3$ bi-cubic maps}]\label{main_rigidity_theorem}
  For any pair of  $C^3$-smooth bi-cubic circle maps $f$ and $g$ with the same irrational rotation number $\rho$ of bounded type and the same signature,
  there exists a conjugacy $\phi$ which is $C^{1+\alpha}$-smooth, where $\alpha>0$ depends only on the bound $B$ on the type.
  \end{theorem}

\begin{remark}
We remark that our proofs in the smooth case carry over to the case of smooth bi-critical circle maps of an arbitrary criticality with the assumption that the attractor of renormalization is hyperbolic.
\end{remark}

The proofs of Theorems~\ref{th:rencov2}~and~\ref{thm-rigidity-analytic} will occupy the bulk of the paper, with the punchlines delivered in \S~\ref{subsec-proofs-analytic}. The proofs of  $C^3$-smooth renormalization convergence and rigidity are carried out in \S~\ref{sec:reduction}.

\subsection{Notation}
When we refer to the circle in this paper, we mean the affine manifold $\bbT=\bbR\slash\bbZ$, naturally identified with the unit circle $S^1$ via the exponential map $z\mapsto e^{2\pi i z}$.
We say that $f\colon\bbT\to\bbT$ is a $C^k$-smooth, $k\geq 3$ (respectively analytic) \textit{multicritical circle map}, if it is a $C^k$-smooth (respectively analytic) orientation preserving circle homeomorphism with a finite non-zero number of non-flat critical points. The latter means that for each critical point $c\in\bbT$ there exists $d\in 2\bbN+1$, called the criticality of $c$, such that in some neighborhood of the point $c$, the map $f$ can be represented as
$$
f(x) = f(c) + (\phi(x))^d,
$$
for some local $C^k$-smooth (respectively analytic) diffeomorphism $\phi$, satisfying $\phi(c)=0$. In particular, any analytic orientation preserving homeomorphism $f\colon\bbT\to\bbT$ is either a diffeomorphism or an analytic multicritical circle map.

We will say that two positive numbers $A, B\in\bbR$ are \textit{commensurable} (denoted by $A\asymp B$) if they satisfy 
$$
\frac{A}{C}\le B\le C A,
$$
for some \textit{a priori} constant $C>1$.

\section{Renormalization of multicritical circle maps}
\subsection{Combinatorics of multicritical circle maps}

For a multicritical circle map $f$, we will always assume that the rotation number $\rho(f)$ is irrational, unless otherwise specified. The following was proven by Yoccoz:

\begin{theorem}[\cite{YocDen}]\label{rotation_conjugacy_theorem}
Suppose $f$ is a $C^3$-smooth multicritical circle map with $\rho(f)\not\in\bbQ$. Then the set of all forward iterates of any point $x\in\bbT$ under the dynamics of $f$ is dense in the circle. Hence, $f$ is topologically conjugate to the rigid rotation by the angle $\rho(f)$.
\end{theorem}
It follows from Theorem~\ref{rotation_conjugacy_theorem} that there exists a unique ergodic $f$-invariant measure, which is the pullback of the Lebesgue measure on the circle by a conjugacy to the rigid rotation. We denote this measure by $\mu_f$. 

We define the \textit{marked signature} of a multicritical circle map $f$ with respect to its critical point $c$ to be the $(2N+2)$-tuple
$$
(\rho(f); N; d_0, d_1,\ldots,d_{N-1}; \delta_0,\delta_1, \ldots, \delta_{N-1}),
$$
where $N$ is the number of critical points, $d_j$ is the criticality of the critical point $c_j$, and $\delta_j = \mu_f[c_j,c_{j+1})$, with the convention that the critical points $c_0,\ldots,c_{N-1}$ are listed in the counterclockwise order starting from the point $c_0=c$ and $c_N=c_0$.




From now on, we will always assume that one of the critical points of each multicritical circle map is positioned at zero. 
\begin{definition}\label{signature_def}
The marked signature of a map $f$ with respect to the critical point at zero will be called simply the \textit{signature} of $f$ and will be denoted by $\sigma(f)$.
\end{definition}

Finally, we define the \textit{total criticality} $D(f)$ of a multicritical circle map $f$ as the product of the criticalities of all its critical points:
$$
D(f):= \prod_{j=0}^{N-1} d_j.
$$


\subsection{Multicritical commuting pairs}
First, we recall the notion of a \textit{multicritical commuting pair}. It is a direct generalization of a \textit{critical commuting pair} which was introduced in~\cite{ORSS} in order to define renormalization of circle maps with a single critical point.

For two points $a,b\in\bbR$, by $[a,b]$ we will denote the closed interval with endpoints $a,b$ without specifying their order.

\begin{definition}\label{commut_pair_def}
A \textit{multicritical commuting pair} of class $C^r$, $r\geq 3$ ($C^\infty$, analytic) is an ordered pair of maps $\zeta=(\eta,\xi)$, with the corresponding smoothness, such that the following properties hold:
\begin{enumerate}[(i)]
 \item the maps $\eta$ and $\xi$ are defined on the intervals $I_\eta=[a,0]$ and $I_\xi=[b,0]$ respectively, where $a,b\in\bbR$ are such that $ab<0$. Furthermore,
 $$
 \eta\colon I_\eta\to[a,b], \qquad \xi\colon I_\xi\to[a,b],
 $$
 and
 $$
 \xi(0)=a,\qquad \eta(0)=b.
 $$
 \item\label{com_pairs_ext_property} there exist $C^r$-smooth ($C^\infty$, analytic) extensions $\hat{\eta}\colon \hat I_\eta\to\bbR$ and $\hat{\xi}\colon \hat I_\xi\to\bbR$ of $\eta$ and $\xi$ respectively to some intervals $\hat I_\eta\Supset I_\eta$ and $\hat I_\xi\Supset I_\xi$, such that $\hat{\eta}$ and $\hat{\xi}$ are orientation-preserving homeomorphisms of $\hat I_\eta$ and $\hat I_\xi$ respectively, onto their images, and
 $$
 \hat{\eta}\circ \hat{\xi}=\hat{\xi}\circ\hat{\eta}
 $$
 where both compositions are defined;
 \item\label{com_pairs_crit_property} the maps $\hat{\eta}$ and $\hat\xi$ can have no more than finitely many critical points on the intervals $I_\eta$ and $I_\xi$ respectively, 
   and both maps $\hat\eta$ and $\hat\xi$ have a critical point at $0$.
   For every critical point $x_*$ of $\hat\eta$,
   we have a local decomposition $\hat\eta=\phi\circ Q_m\circ \psi$,
   where $Q_m(x)=x^m$ with an odd critical exponent
   $m=2k+1, k\in\NN$, and $\phi$, $\psi$ are local
   $C^r$-smooth ($C^\infty$, analytic) diffeomorphisms, with $\psi(x_*)=0$; and
   similarly for the critical points of 
   $\hat\xi$; 
 \item $\xi\circ\eta(0)\in I_\eta$.
\end{enumerate}
\end{definition}


\subsection{$C^r$-metric on commuting pairs}


\begin{definition}
For any $0\le r<\infty$, the $C^r$ (pseudo) metric on the space of $C^r$ multicritical commuting pairs is defined as follows: for any two $C^r$-smooth multicritical commuting pairs $\zeta_1=(\eta_1,\xi_1)$ and $\zeta_2=(\eta_2,\xi_2)$, let $A_1$ and $A_2$ be the M\"obius transformations such that for $i=1,2$,
$$
A_i(\eta_i(0))=-1,\qquad A_i(0)=0,\qquad A_i(\xi_i(0))=1.
$$
Then $\dist_{C^r}(\zeta_1,\zeta_2)$ is defined as
$$
\dist_{C^r}(\zeta_1,\zeta_2)=\max \left\{ \left| \frac{\xi_1(0)}{\eta_1(0)}-\frac{\xi_2(0)}{\eta_2(0)}\right|, \|A_1\circ\zeta_1\circ A_1^{-1} - A_2\circ\zeta_2\circ A_2^{-1}\|_r \right\},
$$
where $\|\cdot\|_r$ is the $C^r$-norm for maps in $[-1,1]$ with one discontinuity at the origin.

\end{definition}

The $C^r$ metrics are actually  pseudo-metrics as they are invariant under linear rescalings. They will become metrics on the subspaces of \textit{normalized} multicritical commuting pairs of appropriate smoothness, that is, the ones for which $\xi(0)=1$ (or, equivalently, $I_\eta=[0,1]$).

\subsection{Gluing procedure}
\label{sec:glue}
A multicritical circle map can be associated to a multicritical commuting pair $\zeta=(\eta|I_\eta,\xi|I_\xi)$ as follows. Let $x\in I_\xi$ be any regular point of $\xi$, and glue the endpoints of the interval $I=[x,\xi(x)]$. More precisely, consider the one-dimensional manifold $S$ of the same smoothness as $\zeta$, obtained by using $\xi$ as a local chart in a small neighborhood of $x$ (without critical points of $\xi$).

Consider the map $f\colon I\to I$ defined as follows:
$$
f(x) = \begin{cases}
\eta(x) &\text{if } x\in I_\eta \text{ and }\eta(x)\in I \\
\xi(\eta(x)) &\text{if } x\in I_\eta \text{ and }\eta(x)\not\in I \\
\eta(\xi(x)) &\text{if } x\in I_\xi.
\end{cases}
$$
Let $\tl f\colon S\to S$ be the projection of $f$ to a map of the topological circle $S$, and let $h$ be a diffeomorphism of the same smoothness as $\zeta$, which maps $S\to \bbT$ and sends $0\in S$ to $0\in\bbT$. It is straightforward to check that 
the map $\tl f$ projects to a multicritical circle map 
\begin{equation}
  \label{eq:fzeta}
f_\zeta=h\circ \tl f\circ h^{-1}\colon\bbT\to\bbT
\end{equation}
via  $h$.
Of course, $f_\zeta$ is not defined uniquely, since there is {\it a priori} no canonical way of identifying the manifold $S$ with the circle $\bbT$. Rather, we recover a conjugacy class of multicritical cicle maps from a pair $\zeta$ from this gluing construction, with the same smoothness as the pair $\zeta$ itself.
Nevertheless, it is easy to check that if $\rho(\zeta)\in\bbR\setminus\bbQ$, then the signature of $f_\zeta$ is independent of the particular choice of the map $f_\zeta$, hence, is determined only by the pair $\zeta$. This motivates the following definition:

\begin{definition}
We will say that two multicritical commuting pairs $\zeta_1$ and $\zeta_2$ have the same signature, if their corresponding multicritical circle maps $f_{\zeta_1}$ and $f_{\zeta_2}$, constructed in~(\ref{eq:fzeta}), have the same signatures in the sense of Definition~\ref{signature_def}. We will denote this by writing $\sigma(\zeta_1)=\sigma(\zeta_2)$.
\end{definition}

\begin{remark}
Note that $\sigma(\zeta_1)=\sigma(\zeta_2)$ does not necessarily imply that there is a topological conjugacy between $\zeta_1$ and $\zeta_2$ taking critical points to critical points. This is due to the fact that the gluing procedure does not ``remember'' which of the two identified subintervals of a multicritical commuting pair had a critical point.
\end{remark}

\subsection{Renormalization of multicritical commuting pairs}
\begin{definition}
For a multicritical commuting pair $\zeta = (\eta,\xi)$, its \textit{height} $\chi(\zeta)$ is defined as the positive integer, such that
$$
\eta^{\chi(\zeta)+1}(\xi(0))<0\le \eta^{\chi(\zeta)}(\xi(0)),
$$
provided that such a number exists. If it does not exist, that is, if $\eta$ has a fixed point, then we define $\chi(\zeta):=\infty$.
\end{definition}

\begin{definition}
Let $\zeta=(\eta|_{I_\eta},\xi|_{I_\xi})$ be a multicritical commuting pair with $\chi(\zeta)=a<\infty$ and $\eta^a(\xi(0))\neq 0$. Then the \textit{prerenormalization} $p\cR\zeta$ of $\zeta$ is defined as
$$
p\cR\zeta:= (\eta^a\circ\xi|_{I_\xi},\, \eta|_{[0,\eta^a(\xi(0))]}),
$$
and the \textit{renormalization} $\cR\zeta$ of $\zeta$ is the rescaling of $p\cR\zeta$ by the affine map $A(x)=x/\eta(0)$ which sends the interval $I_\xi$ to $[0,1]$:
$$
\cR\zeta:= A\circ p\cR\zeta\circ A^{-1}.
$$
\end{definition}
For a multicritical commuting pair $\zeta$ we define its \textit{rotation number} $\rho(\zeta)\in [0,1]$ to be equal to the continued fraction $[r_0,r_1,\ldots]$, where $r_j = \chi(\cR^j\zeta)$. In this definition $1/\infty$ is understood as $0$, hence a rotation number is rational if and only if only finitely many renormalizations of $\zeta$ are defined.

Let $G\colon (0,1)\to [0,1)$ be the Gauss map defined by $G\colon x\mapsto \{\frac{1}{x}\}$ or equivalently,
$$
G\colon [r_0, r_1,r_2,\ldots]\mapsto [r_1,r_2,\ldots].
$$
It is straightforward to see that if a  multicritical commuting pair $\zeta$ is $n$-times renormalizable, then
$$
\qquad \rho(\cR^n \zeta) = G^n(\rho(\zeta)).
$$
One can also verify without difficulty that
the rotation number
  $$\rho(\zeta)=\rho(f_\zeta)$$
  where $f_\zeta$ is as in (\ref{eq:fzeta}).
 However, the way we defined $\rho(\zeta)$ has an added advantage of producing a canonical continued fraction expansion when the rotation number is rational.

\subsection{Renormalization of multicritical circle maps}
For each $n\in\bbN$ we define $I_n:=[0, f^{q_n}(0)]\subset\bbT$ to be the arc that contains the point $f^{q_{n+2}}(0)$. Note that the arcs $I_n$ and $I_{n+1}$ lie on the opposite sides of $0$, and
$$
f^{q_n}(I_{n+1})\subset I_n\qquad\text{and}\qquad f^{q_{n+1}}(I_n)\subset I_n\cup I_{n+1}.
$$
Let $\hat I_n, \hat I_{n+1}\subset\bbR$ be the closed intervals, adjacent to zero and projected down to $I_n$ and $I_{n+1}$ respectively by the natural projection $\bbR\to\bbT$. Then the corresponding lifts of the maps $f^{q_{n+1}}$ and $f^{q_n}$ to the maps $\eta\colon \hat I_n\to \hat I_n\cup \hat I_{n+1}$ and $\xi\colon \hat I_{n+1}\to \hat I_n$ form a multicritical commuting pair. Henceforth, we shall abuse notation and simply denote this commuting pair by
$$
\zeta_n = (f^{q_{n+1}}|I_n,\, f^{q_n}|I_{n+1}),
$$
and call it  the $n$-th prerenormalization $p\cR^n f$.
The $n$-th renormalization $\cR^n f$ of $f$ is the rescaling of $p\cR^n f$ by the rescaling $A_n(x)=x/f^{q_n}(0)$ which sends the interval $I_n$ to  $[0,1]$:
$$
\cR^n f:= \big(A_n\circ f^{q_{n+1}}\circ A_n^{-1}|_{[0,1]},\, A_n\circ f^{q_n}\circ A_n^{-1}|_{A_n(I_{n+1})}\big).
$$
In particular, it is easy to see that for a multicritical circle map with an irrational rotation number, the following holds:
$$
\cR^n(\cR^m f) = \cR^{n+m} f,
$$
where on the left we have the $n$-th renormalization of the commuting pair $\cR^m f$, and on the right we have the $(n+m)$-th renormalization of the multicritical circle map $f$.

\subsection{Dynamical partititons and real {\it a priori} bounds}
Let us recall the definition of a dynamical partition of a multicritical circle map or a commuting pair. We will start with the case of a circle map, as it is easier to formulate.

\begin{definition}
  Let $f$ be a multicritical circle map which is at least $n$ times renormalizable. Then the collection of the intervals
  $$\cI_n(f)\equiv\left\{  I_n,f(I_n),\ldots,f^{q_{n+1}-1}(I_n)\right\}\cup
  \left\{  I_{n+1},f(I_{n+1}),\ldots,f^{q_{n}-1}(I_{n+1})\right\}$$
  is called the {\it $n$-th dynamical partition of} $f$.

\end{definition}
It is not hard to verify that the intervals $\cI_n(f)$ indeed form a partition of the circle, with overlaps only at the endpoints. Obviously, for each $t>0$, the partition $\mathcal I_{n+t}(f)$ is a refinement of  $\mathcal I_n(f)$.

The definition of a dynamical partition for a multicritical commuting pair is analogous:
\begin{definition}
 Let $\zeta$ be an  $n$-times renormalizable multicritical commuting pair $\zeta$. 
 Let $\zeta_n=(\eta_n,\xi_n)= p\cR^n\zeta$ be its $n$-th prerenormalization and consider the collection of intervals $\mathcal I_n(\zeta)$ that consists of finite forward orbits of each of the intervals $[\eta_n(0),0]$ and $[0,\xi_n(0)]$ under the dynamics of $\zeta$ until the last iterate before the first return to the interval $[\eta_n(0),\xi_n(0)]$. It is
called the {\it $n$-th dynamical partition of $\zeta$}. 
\end{definition}

\noindent
Again, the intervals  $\mathcal I_n(\zeta)$ form a partition of the interval $[\eta(0),\xi(0)]$, with overlaps only at the endpoints, and $\cI_{n+t}(\zeta)$ is a refinement of $\cI_{n}(\zeta)$ for $t>0$.

Real {\it a priori} bounds is a statement on the geometry of the dynamical partitions; we will formulate them for the case of pairs (which obviously also covers the case of maps). The following version of real bounds was proven by Estevez and de~Faria  in~\cite{Estevez_deFaria_18} in a somewhat greater generality:

\begin{theorem}[{\bf Real {\it a priori} bounds}]\label{real_bounds_thm1}
  For any $B>0$, $D\geq 3$ there exists $C>1$ such that the following holds. 
  For any $C^3$-smooth multicritical commuting pair $\zeta$ with an irrational rotation number of the type bounded by $B$ with total criticality bounded by $D$, there exists $n_0\in \NN$ such that for any $n\geq n_0$ we have:
  \begin{enumerate}
	\item if $I, J\in\mathcal I_n(\zeta)$ are two adjacent intervals, then 
	$$
	C^{-1}|J|\le |I|\le C|J|;
	$$
      \item if $I\in \mathcal I_n(\zeta)$, $J\in \mathcal I_{n+1}(\zeta)$ and $J\subset I$, then $|J|\ge C^{-1}|I|$;
      \item if $I_1\in \cI_n(\zeta)$ is an interval in the forward orbit $[\eta_n(0),0]$ by the dynamics of $\zeta$, and $I_2\in \cI_n(\zeta)$ is some subsequent interval in the same orbit, obtained by applying a corresponding composition $g$ of iterates of $\eta$ and $\xi$, then
        $$g|I_1=\psi_{s+1}\circ p_s\circ \psi_{s}\circ p_{s-1}\circ\cdots\circ p_0\circ\psi_0,$$
        where $p_j=x^{d_j}$ with $d_j\in\{2\NN+1\}$, $\prod d_j\leq D$, and each $\psi_j$ is a diffeomorphism whose distortion is not greater than $C$.
\end{enumerate}
\end{theorem}
\begin{definition}\label{bound_geom_def}
We will sometimes abbreviate the conditions (1)-(3) by saying that {\it the geometry of the dynamical partition $\cI_n(\zeta)$ is bounded by $C$}.
  \end{definition}

\begin{remark}\label{bounded_goemetry_remark}
We note that condition~(2) implies that if $\rho(\zeta)\not\in\bbQ$ and the geometry of all dynamical partitions of $\zeta$ are bounded by the same constant $C$, then the rotation number $\rho(\zeta)$ is of type bounded by $B=B(C)$.
\end{remark}


\section{Holomorphic commuting pairs and complex {\it a priori} bounds}



A holomorphic commuting pair is an appropriate analytic extension of an (analytic) multicritical commuting pair to the complex plane, defined as follows:
\begin{definition}\label{holo_pair_def}
Given an analytic multicritical commuting pair $\zeta=(\eta|_{I_\eta},\xi|_{I_\xi})$, we say that it extends to a \textit{holomorphic commuting pair} $\cH$ if there exist 
three simply-connected $\RR$-symmetric domains $O, U, V\subset\bbC$ whose intersections with the real line are denoted by $I_U=U\cap\bbR$, $I_V=V\cap\bbR$, $I_O=O\cap\bbR$, and a simply connected $\bbR$-symmetric Jordan domain $\Delta$, such that
\begin{enumerate}[(i)]
\item $\overline O,\; \overline U,\; \overline V\subset \Delta$;
 $\overline U\cap \overline V=\{ 0\}\subset O$; the sets
  $U\setminus O$,  $V\setminus O$, $O\setminus U$, and $O\setminus V$ 
  are nonempty, connected, and simply-connected; 
  $I_\eta\subset I_U\cup\{0\}$, $I_\xi\subset I_V\cup\{0\}$;
  and $\overline U\cap\bbR=\overline{I_U}$ and $\overline V\cap\bbR=\overline{I_V}$;

\item the maps $\eta$ and $\xi$ have analytic extensions to $U$ and $V$ respectively, so that $\eta$ is a branched covering map of $U$ onto $(\Delta\setminus\bbR)\cup\eta(I_U)$, and $\xi$ is a branched covering map of $V$ onto $(\Delta\setminus\bbR)\cup\xi(I_V)$; 

\item the maps $\eta\colon U\to\Delta$ and $\xi\colon V\to\Delta$ can be further extended to analytic maps $\hat{\eta}\colon U\cup O\to\Delta$ and $\hat{\xi}\colon V\cup O\to\Delta$, so that the map $\nu=\hat{\eta}\circ\hat{\xi}=\hat{\xi}\circ\hat{\eta}$ is defined in $O$ and is a branched covering of $O$ onto $(\Delta\setminus \bbR)\cup\nu(I_O)$; 

\item\label{analyt_extensions_property} there exist three simply connected domains $\hat U,\hat V,\hat O\subset\bbC$ such that $U\Subset\hat U\Subset\Delta$, $V\Subset \hat V\Subset\Delta$, $O\Subset \hat O\Subset\Delta$, and the maps $\eta$, $\xi$ and $\nu$ extend analytically to the domains $\hat U$, $\hat V$ and $\hat O$ respectively so that all critical points of these extensions have real critical values;

\item\label{k_H_property} 

if $a,b\in\bbR$ are the nonzero endpoints of the intervals $I_V$ and respectively, $I_U$, then $\eta(b)=\xi(0)$ and $\xi^k(a)=\eta(0)$, for some integer $k=k(\cH)\ge 1$, {and $a$ and $b$ are critical points} of $\xi$ and $\eta$ respectively.


\end{enumerate}
\end{definition}

\noindent
We shall call $\zeta$ the {\it (multicritical) commuting pair underlying $\cH$}, and write $\zeta\equiv \zeta_\cH$. When no confusion is possible, we will use the same letters $\eta$ and $\xi$ to denote both the maps of the commuting pair $\zeta_\cH$ and their analytic extensions to the corresponding domains $U$ and $V$. The interval $\overline{I_U\cup I_V}$ will be called the \textit{large dynamical interval} of $\cH$.

\begin{remark}
Firstly, according to our definition, the maps $\eta$, $\xi$ and $\nu$ that constitute a holomorphic commuting pair, are allowed to have critical points in the interior of their domains $U$, $V$ and $O$ respectively. Also, the domains $U$, $V$ and $O$ are not necessarily Jordan domains as they may have slits in the complex plane that are preimages of the slits along the real line. This differs from the classical definition of holomorphic pairs that are extensions of unicritical commuting pairs (cf.~\cite{DeFar}).
\end{remark}

The sets $\Omega_\cH=O\cup U\cup V$ and $\Delta\equiv\Delta_\cH$ will be called \textit{the domain} and \textit{the range} of a holomorphic pair $\cH$. We will sometimes write $\Omega$ instead of $\Omega_\cH$, when this does not cause any confusion. 

We can associate to a holomorphic pair $\cH$  a single-valued piecewise defined map $S_\cH\colon\Omega\to\Delta$:
\begin{equation}
  \label{eq:shadow1}
S_\cH(z)=\begin{cases}
\eta(z),&\text{ if } z\in U,\\
\xi(z),&\text{ if } z\in  V,\\
\nu(z),&\text{ if } z\in\Omega\setminus(U\cup V).
\end{cases}
\end{equation}
De~Faria \cite{DeFar} calls $S_\cH$ the {\it shadow} of the holomorphic pair $\cH$. Property~\ref{analyt_extensions_property} from the definition of a holomorphic pair (Definition~\ref{holo_pair_def}) implies that the shadow $S_\cH$ has a piecewise analytic extension $\overline S_\cH\colon\overline\Omega\to\overline\Delta$ to the boundary of $\Omega$, defined as follows:
\begin{equation}
  \label{eq:shadow2}
\overline S_\cH(z)=\begin{cases}
S_\cH(z),&\text{ if } z\in \Omega,\\
\eta(z),&\text{ if } z\in \overline U\setminus\Omega,\\
\xi(z),&\text{ if } z\in  \overline V\setminus\Omega,\\
\nu(z),&\text{ if } z\in\partial O\setminus (\overline U\cup\overline V).
\end{cases}
\end{equation}

We say that two holomorphic pairs $\cH_1$ and $\cH_2$ are conjugate (smoothly, analytically, $K$-quasiconformally)
if there exists a real symmetric homeomorphism $h\colon\Delta_{\cH_1}\to \Delta_{\cH_2}$  with the appropriate regularity that conjugates the shadows $S_{\cH_1}$ and $S_{\cH_2}$.

To discuss convergence in the space of holomorphic pairs, we use Carath\'eodory topology on the space of analytic maps of marked  domains~\cite{McM-ren1}.
We associate to each $\cH$ three triples of such maps
$$
(U, \, \xi(0)/2, \, \eta), \qquad (V, \, \eta(0)/2, \, \xi), \qquad (O, \, 0, \, \nu)
$$
and denote the space of holomorphic pairs with the corresponding product topology by $\mathbf H$.
By Carath\'eodory convergence (or simply convergence) of holomorphic pairs we will mean convergence in this space.

Let the \textit{modulus} $\mod(\cH)$ of a holomorphic pair $\cH=(\eta,\xi,\nu)$ be the supremum of the moduli of all annuli $A\subset\Delta$ that separate $\bbC\setminus\Delta$ from $\overline\Omega$. It will also be convenient to use the following notation: $J_\eta:= I_U = U\cap\bbR$ and $J_\xi:= I_V=V\cap\bbR$, where $U$ and $V$ are the domains of the respective maps $\eta$ and $\xi$.  

\begin{definition}\label{Hm_def}
For $m\in(0,1)$ let $\mathbf H(m)\subset\mathbf H$ denote the space of all holomorphic commuting pairs $\cH=(\eta|U,\xi|V,\nu|O)$ with the following properties:
\begin{enumerate}
	\item\label{mod_property} $\mod(\cH)\ge m$;
	\item $\diam(\Delta_\cH)\le 1/m$;
	\item\label{min_interval_property} $\min\{|I_\eta|,\, |I_\xi|\}\ge m$, where $I_\eta = [0,\xi(0)]$ and $I_\xi = [\eta(0),0]$;
	\item\label{Poincare_property} $\Omega_\cH$ is contained in the Poincar\'e neighborhood $D_m(\Omega_\cH\cap\bbR)$ that consists of all points $z\in\bbC$, from which the interval $\Omega_\cH\cap\bbR$ is seen under the angle greater than $m$; 
	\item\label{k_property} $k(\cH)\le 1/m$, where $k(\cH)$ is the same as in property~\ref{k_H_property} of Definition~\ref{holo_pair_def};
	\item\label{extension_property} for every $z\in\Omega_\cH$, there exists a neighborhood $U_z\ni z$, such that $U_z$ contains a round disk of radius $m^2$, centered at $z$, and the shadow $S_\cH$ extends from $z$ to an analytic branched covering $\phi$ of $U_z$ onto its image with topological degree of $\phi$ not exceeding $1/m$; 
	\item\label{real_bound_property} for each $\theta \in \{\eta,\xi\}$ there exists $s\in\bbN$, such that the map $\theta|J_\theta$ can be represented as
	$$
	\theta|J_\theta = \psi_{s+1}\circ p_s\circ\psi_s\circ p_{s-1}\circ\dots\circ\psi_1\circ p_0\circ\psi_0,
	$$
	where $p_j(x)=x^{d_j}$ for an odd integer $d_j\ge 3$, and each $\psi_j$ is a diffeomorphism of the interval 
	$$
	L_j=p_{j-1}\circ\psi_{j-1}\circ p_{j-2}\circ\dots \circ\psi_1\circ p_0\circ\psi_0(J_\theta)
	$$
	that extends to a conformal map of its neighborhood 
	$$
	N_m(L_j) = \{z\in\bbC\mid \inf_{w\in L_j}|z-w|<m|L_j|\}.
	$$
\end{enumerate}
\end{definition}

Let the \textit{degree} of a holomorphic pair $\cH$ denote the maximal topological degree of the covering maps constituting the pair. Denote by $\mathbf H^D(m)$ the subset of $\mathbf H(m)$ consisting of pairs whose degree is not greater than $D$. 

The following lemmas summarizes some basic properties of holomorphic pairs from $\mathbf H^D(m)$.

\begin{lemma}\label{basic_holo_pair_lemma}
For each $D\ge 3$ and $m\in (0,1)$, there exists a real number $r = r(m, D)>0$, such that for any $\cH=(\eta|U,\xi|V,\nu|O)\in \mathbf H^D(m)$ 
and any round disk $D'$ of radius $r$, centered at a point on one of the three intervals $J_\eta$, $J_\xi$ or $O\cap\bbR$, the corresponding images $\eta(D')$, $\xi(D')$ and $\nu(D')$ are defined and contained in $\Delta_{\cH}$.
In particular, the domain $\Omega_{\cH}$ contains the disk of radius $r$, centered at zero.
\end{lemma}
\begin{proof}
First, observe that according to properties~(\ref{mod_property}) and~(\ref{min_interval_property}) of Definition~\ref{Hm_def}, there exists a real number $r_1=r_1(m)>0$, such that the range $\Delta_{\cH}$ contains the $r_1$-neighborhood of the interval $J_\eta\cup J_\xi$.
Next, note that the number $s$ from property~(\ref{real_bound_property}) of Definition~\ref{Hm_def} is bounded from above by a function of the degree $D$, hence, property~(\ref{real_bound_property}) of Definition~\ref{Hm_def} and the Koebe Distortion Theorem, applied to the maps $\eta$, $\xi$ and $\nu=\eta\circ\xi$, imply existence of the number $r=r(m, D)$.
\end{proof}

The following is an easy generalization of Lemma~2.17 of~\cite{Yamp-towers}:

\begin{lemma}\label{H_compactness_lemma}
For each $D\ge 3$ and $m\in (0,1)$, the space $\mathbf H^D(m)$ is sequentially compact.
\end{lemma}
\begin{proof}
Let $\cH_1,\cH_2,\dots$ be a sequence of holomorphic pairs from $\mathbf H^D(m)$, where $\cH_j=(\eta_j|U_j,\xi_j|V_j, \nu_j|O_j)$.
According to Lemma~\ref{basic_holo_pair_lemma}, for each $j\ge 1$, the domain $U_j$ contains a disk of radius $r$, centered at $\xi_j(0)/2$. Then Theorem~5.2 from~\cite{McM-ren1} implies that the sequence of pointed domains $(U_j, \xi_j(0)/2)$ is precompact, hence has a subsequence that Carath\'eodory converges to a pointed domain $(U,u)$. Furthermore, Lemma~\ref{basic_holo_pair_lemma} implies that $\overline{U\cap\bbR}$ is the Hausdorff limit of the intervals $\overline{J_{\eta_j}}$ from the subsequence. Possibly, extracting another subsequence, we ensure that the sequence of maps $\eta_j$ converges to a map $\eta$ on $U$.

Treating the maps $\xi_j$ and $\nu_j$ in a similar way, we ensure that a subsequence of the sequence $\cH_1,\cH_2,\dots$ converges to a holomorphic pair $\cH$. It is obvious that (possibly after extracting a further subsequence) all properties of Definition~\ref{Hm_def} still hold for $\cH$, hence the space $\mathbf H^D(m)$ is sequentially compact.
\end{proof}

Below is another useful property of holomorphic pairs from $\mathbf H^D(m)$.
\begin{lemma}\label{strip_lemma}
For each $D\ge 3$ and $m\in (0,1)$, there exists a real number $\varepsilon = \varepsilon(m, D)>0$, such that for any $\cH\in \mathbf H^D(m)$ and any $z\in\Omega_\cH$, satisfying $|\Im(z)|<\varepsilon$, we have $S_\cH(z)\in\Omega_\cH$.
\end{lemma}
\begin{proof}
Assume that the statement is false. Then there exists a sequence of holomorphic pairs $\cH_1,\cH_2,\dots\in\mathbf H^D(m)$, for which the corresponding numbers $\varepsilon_1,\varepsilon_2,\dots$ decrease to zero. Due to Lemma~\ref{H_compactness_lemma}, this sequence has a subsequence that converges to some $\cH\in\mathbf H^D(m)$.
	
According to properties~(\ref{Poincare_property}) and~(\ref{real_bound_property}) of Definition~\ref{Hm_def}, and the Koebe Distortion Theorem (same argument as in the proof of Lemma~\ref{basic_holo_pair_lemma}), there exists a positive real number $\rho_0>0$ and a continuous function $\rho\colon(0,\rho_0)\to\bbR^+$ with $\rho(s)\to 0$ as $s\to 0^+$, such that if $s=|\Im(z)|<\rho_0$, then $|\Im(S_{\cH_j}(z))|<\rho(s)$, for each index $j$ and $z\in\Omega_{\cH_j}$. Now, taking into account the dynamics of real pairs, this implies that the domains $\Omega_{\cH_j}$ of the maps from the converging subsequence either pinch on the real line, or at least one of the endpoints of the interval $\Omega_\cH\cap\bbR$ is a fixed point of the limiting map $S_\cH$. The first possibility is ruled out by Lemma~\ref{basic_holo_pair_lemma}. The second possibility cannot happen as well, since otherwise $\cH$ is not a holomorphic pair.
\end{proof}

Finally, we can refine the statement of real {\it a priori} bounds for pairs underlying elements of
 $\mathbf H^D(m)$:

\begin{theorem}[{\bf Real {\it a priori} bounds for $\mathbf H^D(m)$}]\label{real_bounds_thm}
  For any pair of integers $D\ge 3$, $B\ge 1$ and a real number $m\in (0,1)$, there exists a real number $C=C(m,D, B)>1$, such that for any holomorphic pair $\cH\in\mathbf H^D(m)$ with an irrational rotation number of type bounded by $B$, the following holds. For the pair $\zeta$ underlying $\cH$, the geometry of the dynamical partition $\mathcal I_n(\zeta)$ is bounded by $C$  for each level  $n\ge 0$.
  \end{theorem}
%
Note, that in the general  Theorem~\ref{real_bounds_thm1} the number $n_0$ depends on the pair $\zeta$.
In contrast, in $\mathbf H^D(m)$ we can select $n_0=0$, by compactness.

For every pair of positive integers $B>0$ and $D\ge 3$, let $\mathcal M_{D,B}$ denote the set of all analytic multicritical circle maps with a critical point at zero, an irrational rotation number of combinatorial type bounded by $B$ and the total criticality not greater than $D$. 

The following Theorem was proven in~\cite{Est_Smania_Yampolsky_2020}:
\begin{theorem}[\textbf{Complex bounds for bounded type}]\label{ComplexBounds_theorem}
For every pair of positive integers $B>0$ and $D\ge 3$, there exists a real constant $m=m(B, D)>0$, such that for any multicritical circle map 
$f\in\mathcal M_{D,B}$, there exists $N=N(f)\in\bbN$ with the property that for any $n\ge N$, the renormalization $\cR^n f$ extends to a holomorphic commuting pair $\cH_n\in \mathbf H^D(m)$. Furthermore, $\cH_n$ can be chosen so that its range $\Delta_{n}$ is a Euclidean disk.
\end{theorem}

One of the consequences of Theorem~\ref{ComplexBounds_theorem} is the following statement on quasiconformal conjugacy:

\begin{theorem}\label{qc_conjugacy_theorem}
For every pair of positive integers $B>0$ and $D\ge 3$, let $m=m(B, D)>0$ be the same as in Theorem~\ref{ComplexBounds_theorem}. Then there exists a real number $K = K(B, D)\ge 1$ with the following property: for any two functions $f, g\in\mathcal M_{D,B}$ with the same signature, there exists $L= L(f,g)\in\bbN$, such that for any $n\ge L$, the renormalizations $\cR^nf$ and $\cR^ng$ extend to two $K$-quasiconformally conjugate holomorphic commuting pairs from $\mathbf H^D(m)$.
\end{theorem}
The proof follows the usual route of a pull-back argument to promote the quasi-symmetric conjugacy between renormalizations \cite{Estevez_deFaria_18} to a quasiconformal one. Taking sufficiently deep renormalizations ensures that the holomorphic pair extensions  of $\cR^nf$ and $\cR^ng$ have matching branched covering structures (this is, of course, automatic in the unicritical case \cite{DeFar,Ya1}). The pull-back argument itself is completely standard, and will be left to the reader. 

\begin{remark}\label{gluing_remark}
Given an analytic pair $\zeta$ with an irrational rotation number, let $f_\zeta$ be an arbitrary circle map, constructed by the gluing procedure (\ref{eq:fzeta}). Then there exists a real-symmetric equatorial neighborhood $W\subset\bbC\slash\bbZ$, such that if the $n$-th prerenormalization $p\cR^nf_\zeta$ extends to a holomorphic pair $\cH_f$ with $\Delta_{\cH_f}\subset W$, then the prerenormalization $p\cR^n\zeta$ also extends to a holomorphic pair $\cH_\zeta$, and the pairs $\cH_f$ and $\cH_\zeta$ are conformally conjugate. In particular, this implies that Theorem~\ref{ComplexBounds_theorem} and Theorem~\ref{qc_conjugacy_theorem} can be restated in an identical form for multicritical commuting pairs instead of the multicritical circle maps.
\end{remark}

%% file: holpairs.tex
\section{Expansion of hyperbolic metric for the dynamics of holomorphic pairs}


We collect the basic results on the hyperbolic expansion by the iterates of holomorphic pairs. The corresponding statements for polynomial-like maps are in the now classical work of McMullen~\cite{McM-ren1}. 

For a hyperbolic Riemann surface $X$, let $\dist_X$ denote the distance with respect to the hyperbolic metric on $X$. For a holomorphic map $f\colon X\to Y$ between two hyperbolic Riemann surfaces $X$ and $Y$, let $\|f'(z)\|_{X,Y}$ denote the norm of the derivative of $f$ at $z\in X$, measured with respect to the hyperbolic metric in $X$ and $Y$. We will omit the subscripts and simply write $\|f'(z)\|$ when this does not cause any confusion. If $X$ and $Y$ are both subsets of the complex plane, then $|f'(z)|$ will denote the norm of the derivative, measured in the standard Euclidean metric on $\bbC$.

The following lemma is standard in the theory of hyperbolic Riemann surfaces (see Proposition~4.9 from~\cite{McM-ren2} 
for a proof). 
\begin{lemma}\label{hyper_inclusion_lemma}
There exists a universal positive function $C(s)<1$ decreasing to zero as $s$ decreases to zero, such that the following holds. If $X\subsetneq Y$ are two distinct hyperbolic Riemann surfaces and $f\colon X\to Y$ is the standard inclusion, then for any $z\in X$ we have
$$
\|f'(z)\|_{X,Y}\le C(\dist_Y(z, \,Y\setminus X)).
$$
\end{lemma}

The following two lemmas are direct adaptations of Lemma~4.10 parts (I) and (II) of \cite{Yamp-towers} (where the corresponding statements are given for the unicritical case). The proofs are identical and will be omitted.

\begin{lemma}\label{basic_expansion_lemma}
Let $\cH$ be a holomorphic pair with domain $\Omega$ and range $\Delta$. Then for any $z\in\Omega\setminus\bbR$, such that $S_\cH(z)\not\in\bbR$, we have $\|S_\cH'(z)\|>1$ in the hyperbolic metric of $\Delta\setminus\bbR$.

Furthermore, for any $\varepsilon>0$, there exists $r>1$, such that if $|\Im(z)|>\varepsilon$ and $S_\cH(z)\not\in\bbR$, then $\|S_\cH'(z)\|>r$ in the hyperbolic metric of $\Delta\setminus\bbR$.
\end{lemma}

\ignore{
\begin{proof}
Define the set $S$ as the preimage $S:= \S_\cH^{-1}(\bbR)$. Then $S_\cH\colon (\Omega\setminus S)\to (\Delta\setminus\bbR)$ is a covering map, hence a local isometry between $(\Omega\setminus S)$ and $(\Delta\setminus\bbR)$. On the other hand, according to Lemma~\ref{hyper_inclusion_lemma}, the inclusion $i\colon (\Omega\setminus S)\to (\Delta\setminus\bbR)$ satisfies
$$
\|i'(z)\|_{(\Omega\setminus S), (\Delta\setminus\bbR)}\le \rho(z)<1,
$$
for every $z\in (\Omega\setminus S)$, where 
$$
\rho(z) = C(\dist_{(\Delta\setminus\bbR)}(z,\, (\Delta\setminus\bbR)\setminus (\Omega\setminus S)).
$$
Thus, according to the Chain Rule, we have
$$
\|S_\cH'(z)\|\ge \frac{1}{\rho(z)}>1.
$$

Finally, it follows from Lemma~\ref{hyper_inclusion_lemma} that if $|\Im(z)|>\varepsilon$, then $\rho(z)<\rho_1(\varepsilon)<1$, for some function $\rho_1$. Hence, one can take $r = 1/\rho_1(\varepsilon)>1$, which completes the proof.
\end{proof}

}

\begin{lemma}[{\bf Definite expansion}]\label{expansion_at_escaping_points_lemma}
For every integer $D\ge 3$ and every real $m>0$, there exists a real number $\lambda=\lambda(m, D)>1$, such that the following holds. 
Let $\cH$ and $\cH'$ be two holomorphic pairs, satisfying
\begin{itemize}
\item $\zeta_{\cH'}= p\cR^n\zeta_\cH$ for some $n\in\bbN$;
\item $\Delta_{\cH'}\subset\Delta_\cH$;
\item an appropriate affine rescaling of $\cH'$ is contained in $\mathbf H^D(m)$;
\end{itemize}
Then for any $z\in\Omega_{\cH'}$, such that $S_{\cH'}(z)\not\in\Omega_{\cH'}$, we have 
$$
\|S_{\cH'}'(z)\|\ge \lambda
$$
in the hyperbolic metric of $\Delta_{\cH}\setminus\bbR$ (not to be confused with the hyperbolic metric of $\Delta_{\cH'}\setminus\bbR$).
\end{lemma}

\begin{definition}
Given a holomorphic pair $\cH$, its \textit{filled Julia set} $K(\cH)$ is defined as the set of all points in $\overline{\Omega_\cH}$, whose forward orbits under the dynamics of $\overline S_\cH$ do not escape $\overline{\Omega_\cH}$. The Julia set of $\cH$ is denoted by $J(\cH)$ and defined as the boundary of $K(\cH)$.
\end{definition}

\begin{lemma}[{\bf Expansion on $J(\cH)$}]\label{expansion_on_J_lemma}
Let $\mathcal H$ be a holomorphic pair with domain $\Omega$ and range $\Delta$. Let $Q\subset \overline \Omega$ be the set of all points whose orbits under the dynamics of $\overline S_\cH$ land on the interval $\Omega\cap\bbR$. Then $Q$ is dense in $J(\cH)$, and for any $z\in J(\cH)\setminus Q$,
$$
\|(S_\cH^n)'(z)\|\to\infty
$$
in the hyperbolic metric of $\Delta\setminus\bbR$ as $n\to\infty$.
\end{lemma}

\begin{proof}
Define $I:=\overline\Omega\cap\bbR$. For each $n\in\bbN$, consider the set
$$
Q_n:= I\cup \overline S_\cH^{-1}(I)\cup \overline S_\cH^{-2}(I)\cup\ldots\cup \overline S_\cH^{-n}(I).
$$
Then $S_\cH^n\colon\Omega\setminus Q_n\to\Delta\setminus\bbR$ is a local isometry with respect to the hyperbolic metrics of $\Omega\setminus Q_n$ and $\Delta\setminus\bbR$. Note that 
$$
Q=\bigcup_{n=0}^\infty Q_n,
$$
and $J(\cH)\setminus Q \subset \Omega$, since the points from $\partial\Omega$ either escape $\overline\Omega$ or land on the interval $I$. Hence, if $z\in J(\cH)\setminus Q$, then $z\in\Omega$. Furthermore, if $z\in J(\cH)\setminus Q$ belongs to the closure of $Q$, then the hyperbolic distance 
$\dist_{\Delta\setminus\bbR}(z,Q_n)\to 0$ as $n\to\infty$, so for the inclusion $i_n\colon\Omega\setminus Q_n\to\Delta\setminus\bbR$, we have 
$$
\|i_n'(z)\|\to 0\qquad\text{as }n\to\infty,
$$
where the norm is taken with respect to the hyperbolic metric of $\Delta\setminus\bbR$ in the image and $\Omega\setminus Q_n$ in the domain. Thus, we obtain that $\|(S_\cH^n)'(z)\|\to\infty$.

In order to complete the proof, we need to show that $Q$ is dense in $J(\cH)$. Assume, this is not the case. Then there exists $z_0\in J(\cH)\setminus Q$, which is not an accumulation point of the sets $Q_n$. This implies that there exist positive constants $c_0, r_0>0$, such that for every $n\in\bbN$ and for every $z$ from the ball $B(z_0,r_0)$,
$$
\|i_n'(z)\|>c_0.
$$
Then it follows that for every $z\in B(z_0,r_0)$ and every $n\in\bbN$, such that $S_\cH^n(z)$ is defined, we have
\begin{equation}\label{Sn_upper_bound_norm_eq}
\|(S_\cH^n)'(z)\|<1/c_0
\end{equation}
in the hyperbolic metric of $\Delta\setminus\bbR$. Therefore, there exists a constant $c_1>0$, such that 
$$
|(S_\cH^n)'(z)|<c_1,
$$
for all $z$ and $n$ as above. Thus, for every $0<r<r_0$ and every $n\in\bbN$, such that $S_\cH^n$ is defined on the ball $B_r = B(z_0,r)$, we have
\begin{equation}\label{S_Hn_Br_size_eq}
\diam(S_\cH^n(B_r))\le 2c_1 r.
\end{equation}
On the other hand, for every $n\in\bbN$, let $z_n:= S_\cH^n(z_0)$. Then it follows from~(\ref{Sn_upper_bound_norm_eq}), Lemma~\ref{basic_expansion_lemma} and the Chain Rule that 
$$
\lim_{n\to\infty} \|S_\cH'(z_n)\|=1,
$$
which implies that the points $z_n$ converge to the real line. Next, from continuity arguments and the dynamics of the map $S_\cH$ on the real line, we conclude that the points $z_n$ converge to the interval $I_\eta\cup I_\xi$ 
which is well inside $\Omega$. Therefore, we obtain from~(\ref{S_Hn_Br_size_eq}) that there exists a ball around $z_0$ that does not escape from $\overline\Omega$ under iterates of $\overline S_\cH$. The latter contradicts the fact that $z_0\in J(\cH)$.
\end{proof}

\section{Filled Julia sets of infinitely renormalizable pairs have  empty interior}

Next, we restrict our consideration to the case of holomorphic pairs with irrational rotation numbers. Our main goal here is the following proposition: 

\begin{proposition}\label{K_is_J_proposition}
Let $\cH$ be a holomorphic pair with an irrational rotation number of bounded type. Then its filled Julia set $K(\cH)$ has no interior points and coincides with the Julia set $J(\cH)$, which is the closure of all preimages of the interval $\Omega_\cH\cap\bbR$ under the dynamics of $\overline S_\cH$.
\end{proposition}

First, we observe that it is sufficient to show that the open set $K(\cH)\setminus J(\cH)$ has no connected components, hence is empty. The second assertion of Proposition~\ref{K_is_J_proposition} then follows immediately from Lemma~\ref{expansion_on_J_lemma}.

We also note that Proposition~\ref{K_is_J_proposition} is a generalization of the analogous statement for unicritical holomorphic pairs. The latter was proven for pairs with a cubic critical point in~\cite{DeFar} (see also~\cite{Yamp-towers}) by constructing a topological conjugacy between $\cH$ and a holomorphic pair whose components are restrictions of certain entire functions which are quasiconformally rigid. In particular, these entire functions have no wandering domains. The same construction was carried out for general unicritical holomorphic pairs in \cite{Yampolsky_unicritical}, where quasiconformally rigid Blaschke product models were constructed. In the case of several critical points, constructing quasiconformally rigid models either in the family of rational maps or entire functions is a challenge (cf. \cite{Yam-bicubic} for the case of two cubic critical points). However, proving that there are no wondering domains can still be done by a modification of the original argument, which we present below.

The proof of Proposition~\ref{K_is_J_proposition} will be split into several steps.

\begin{lemma}\label{Fatou_convergence_to_R_lemma}
Let $\cH$ be a holomorphic commuting pair with the underlying commuting pair $(\eta,\xi)$, and let $U\subset K(\cH)\setminus J(\cH)$ be a nonempty connected component of $K(\cH)\setminus J(\cH)$. Then for any $z\in U$, we have $S_\cH^n(z)\to [\eta(0),\xi(0)]$ in the Euclidean metric on $\bbC$, as $n\to\infty$.
\end{lemma}
\begin{proof}
Assume that the sequence of points $\{S_\cH^n(z)\}_{n\in\bbN}$ does not converge to the interval $[\eta(0),\xi(0)]$. Then it does not converge to the real line, hence, according to Lemma~\ref{basic_expansion_lemma}, we have
$$
\|(S_\cH^n)'(z)\|\to\infty \qquad\text{as }n\to\infty
$$
in the hyperbolic metric of $\Delta_\cH\setminus\bbR$. After extracting a subsequence of indices $\{n_k\}$, such that the points $S_\cH^{n_k}(z)$ stay some fixed distance away from the real line, we conclude that 
$$
|(S_\cH^{n_k})'(z)|\to\infty \qquad\text{as }k\to\infty.
$$
The latter implies that a small disk $D\Subset U$, centered at $z$, gets arbitrarily large under the iterates of $S_\cH^{n_k}$, hence $S_\cH^{n_k}(D)\not\subset\Omega_\cH$, for some $k\in\bbN$, which is a contradiction to the assumption that $U\subset K(\cH)$.
\end{proof}

\begin{lemma}\label{no_periodic_components_lemma}
Let $\cH$ be a holomorphic commuting pair with an irrational rotation number. Then the set $K(\cH)\setminus J(\cH)$ does not have any nonempty periodic connected components.
\end{lemma}
\begin{proof}
(c.f.~\cite{DeFar}.) 
Let $(\eta,\xi)$ be the underlying commuting pair of $\cH$, and assume that $U\subset K(\cH)\setminus J(\cH)$ is a nonempty periodic connected component of $K(\cH)\setminus J(\cH)$. Then, according to Lemma~\ref{Fatou_convergence_to_R_lemma}, the intersection $\partial U\cap [\eta(0),\xi(0)]$ is nonempty. If this intersection consists of a single point, then this point must be periodic for the commuting pair $(\eta,\xi)$, which is not possible, since the pair $(\eta,\xi)$ has an irrational rotation number. If the intersection $\partial U\cap [\eta(0),\xi(0)]$ contains at least two distinct points, then there exists an integer $n\in\bbN$, such that the intersection $\partial (S_\cH^n(U))\cap [\eta(0),\xi(0)]$ has points both to the left and to the right from zero. For topological reasons, the latter implies that the intersection $S_\cH^n(U)\cap S_\cH^{-1}([\eta(0),\xi(0)])$ is nonempty, therefore, $S_\cH^n(U)\cap J(\cH)\neq\varnothing$, which leads to a contradiction.
\end{proof}

\begin{lemma}\label{model_family_lemma}
	For any signature $\sigma$ there exists a rational map $F\colon\hat\bbC\to\hat\bbC$ that restricts to a homeomorphism of the unit circle, so that the lift $f\colon\bbT\to\bbT$ of this homeomorphism via the projection $z\mapsto e^{2i\pi z}$, is an analytic multicritical circle map with a critical point at zero and the signature $\sigma$.
\end{lemma}
\begin{proof}
	
	Given a criticality $d\in 2\bbN+1$, we will say that a rational map $B\colon\hat{\bbC}\to\hat{\bbC}$ is \textit{basic}, if it restricts to a homeomorphism of the unit circle $S^1\subset\bbC$ with a fixed point at $z=1$, which is also the unique critical point of $B|_{S_1}$. It is well known that such a map is necessarily a Blaschke fraction. Existence of such maps with arbitrary criticality $d\in 2\bbN+1$ was established in~\cite{Yampolsky_unicritical}.
	
	Let $N\ge 1$ be an arbitrary integer. For any $N$-tuple of basic Blaschke fractions $B_0,\dots,B_{N-1}$ and any $N$-tuple of real numbers $\theta_0,\theta_1,\ldots,\theta_{N-1}\in [0,1]$, the map 
	$$
	F = e^{-2\pi i\theta_{N-1}}B_{N-1}\circ e^{-2\pi i\theta_{N-2}} B_{N-2}\circ\ldots\circ 
	e^{-2\pi i\theta_{0}} B_0
	$$
	is a Blaschke fraction whose restriction to the unit circle lifts via the projection $z\mapsto e^{2\pi i z}$ to a multicritical circle map with at most $N$ critical points.
	
	Fix an arbitrary irrational rotation number $\rho\in (0,1)$. We will show that for any signature 
	$$
	\sigma = (\rho; N; d_0, d_1,\ldots,d_{N-1}; \delta_0,\delta_1, \ldots, \delta_{N-1}),
	$$
	where $d_j$ are the criticalities of the corresponding basic maps $B_j$, there exist real numbers $\theta_0,\theta_1,\ldots,\theta_{N-1}\in [0,1]$, such that the map $F$ lifts to a multicritical circle map with signature $\sigma$.
	
	According to the definition of the signature, the $N$-tuple $(\delta_0,\ldots,\delta_{N-1})$ is an arbitrary point of the $(N-1)$-simplex
	$$
	S_N = \{(x_0,\ldots,x_{N-1})\in\bbR^N \mid \forall j\in [0,N-1], x_j>0,\text{ and } \sum_{j=0}^{N-1}x_j=1\}.
	$$
	The boundary $\partial S_N = \overline S_N\setminus S_N$ corresponds to ``degenerate'' signatures, where two or more critical points collapse into one. In the course of this proof we will consider them as valid signatures as well. For every $j=0,\dots,N-1$, let
	$$
	S_N^j = \overline S_N\cap \{(x_0,\ldots,x_{N-1})\in\bbR^N \mid x_j=0\}
	$$
	denote the $(N-2)$-dimensional facets of the simplex $S_N$.
	
	First, note that due to monotonic and continuous dependence of the rotation number on the parameters, for each tuple $(\theta_0,\ldots,\theta_{N-2})$, there exists $\theta_{N-1}\in[0,1)$ such that $\rho(F|_{S^1})=\rho$. Since $\rho\not\in\bbQ$, such a number $\theta_{N-1}$ is unique.
	Assuming that $\theta_{N-1} = \theta_{N-1}(\theta_0,\ldots,\theta_{N-2})$ is always chosen so that $\rho(F|_{S^1})=\rho$, it follows that the homeomorphism $h\colon S^1\to \bbT$ that takes $1$ to $0$ and conjugates $F|_{S^1}$ with the rigid rotation, depends continuously on the remaining parameters $\theta_0,\ldots, \theta_{N-2}$. 
	Therefore, the map $\Lambda\colon [0,1]^{N-1}\to\overline{S}_N$ taking the $(N-1)$-tuple $(\theta_0,\ldots, \theta_{N-2})$ to the $N$-tuple $(\delta_0,\dots,\delta_{N-1})$ of distances in the signature of $F|_{S^1}$, is continuous at all points $(\theta_0,\ldots, \theta_{N-2})$, for which $\Lambda(\theta_0,\ldots, \theta_{N-2})\in S_N$.
	
	Let $T_N\subset [0,1]^{N-1}$ be the connected component of the preimage $\Lambda^{-1}(S_N)$, such that $(0,\dots,0)\in \overline T_N$. By the continuity property of $\Lambda$, we have
	$$
	\Lambda(\partial T_N)\subset \partial S_N.
	$$
	Of course, the set $T_N$ and the map $\Lambda$ depend on the choice of the basic maps $B_j$. We will prove by induction on $N$ that for any such choice, the set $T_N$ is homotopic to a point and $\Lambda(\overline T_N)=\overline S_N$. This will complete the proof of the Lemma.

        The base of induction for $N=1$ is evident.
	Assume that the above statement holds for all $N<k$, where $k\ge 2$. We will prove it for $N=k$.
	
	For $j=0,\dots, N-2$, consider the sets 
	$$
	U_j = \{(\theta_0,\dots,\theta_{N-2})\in\partial T_N\mid \theta_j=0\},
	$$
	and also define 
	$$
	U_{N-1} = \overline{\partial T_N\setminus \cup_{j=0}^{N-2} H_j}.
	$$
	It follows from the definition of the map $\Lambda$ that $U_{N-1}$ is the maximal subset of $\partial T_N$, such that $\Lambda(U_{N-1})\subset S_N^{N-1}$, or equivalently, $U_{N-1}$ consists of all $(N-1)$-tuples $(\theta_0,\dots,\theta_{N-2})\in\partial T_N$, for which $e^{2\pi i \theta_{N-1}}F(1) = 1$. Due to monotonic dependence of the expression $e^{2\pi i \theta_{N-1}}F(1)$ on each individual parameter $\theta_0,\dots,\theta_{N-2}$, it follows that $U_{N-1}$ is a graph of a continuous function over each of the sets $U_0,\dots,U_{N-2}$. The latter implies that the sets $U_0,\dots,U_{N-1}$ and $T_N$ are homotopic to a point, and $T_N$ can be viewed as an $(N-1)$-dimensional simplex whose $(N-2)$-dimensional facets are $U_0,\dots, U_{N-1}$.
	
	Since composition of basic Blaschke fractions is again a basic Blaschke fraction, it follows that for each $j=0,\dots, N-2$, the facet $U_j$ coincides with the set $T_{N-1}$ for some $N-1$ basic Blaschke fractions. Hence, by the inductive assumption,
	\begin{equation}\label{lambda_U_eq}
	\Lambda(\overline U_j) = S_N^j, \qquad\text{for }j=0,\dots,N-2.	
	\end{equation}
	In particular, since $T_N$ has a structure of a simplex and $\Lambda$ maps facets of $T_N$ to corresponding facets of $S_N$, this implies that $\Lambda$ takes the $(N-3)$-dimensional boundary of $\partial U_{N-1}$ surjectively onto the boundary of $S_N^{N-1}$. Since the map $\Lambda$ is continuous, and both $S_N^{N-1}$ and $U_{N-1}$ are homotopic to a point, it follows that $\Lambda(\overline U_{N-1}) = S_N^{N-1}$. Combining this with~(\ref{lambda_U_eq}), we conclude that $\Lambda(\partial T_N)=\partial S_N$. Finally, since $T_N$ and $S_N$ are both homotopic to a point, continuity of $\Lambda$ implies that $\Lambda(\overline T_N) = \overline S_N$.

\end{proof}

\begin{lemma}\label{no_wandering_lemma}
Let $f\colon\bbT\to\bbT$ be an arbitrary analytic multicritical circle map obtained in Lemma~\ref{model_family_lemma}. Assume that $n\in\bbN$ is such that the pre-renormalization $p\cR^nf$ extends to a holomorphic pair $\cH$. Then the set $K(\cH)\setminus J(\cH)$ has no wandering components. Furthermore, each point $z\in\Omega_\cH\setminus J(\cH)$ escapes to the annulus $\Delta_{\cH}\setminus\Omega_\cH$ under the dynamics of the shadow $S_\cH$.
\end{lemma}
\begin{proof}
Let $F\colon\hat{\bbC}\to\hat{\bbC}$ be the rational map from Lemma~\ref{model_family_lemma} whose restriction to the unit circle is conjugate to $f$ via the map $h\colon z\mapsto \frac{1}{2i\pi}\log z\quad (\mod 1)$. 
Let $J\subset\hat{\bbC}$ denote the Julia set of the rational map $F$. 

Assume, $U\subset K(\cH)\setminus J(\cH)$ is a wandering component. Then the preimage $h^{-1}(U)$ is contained in the Fatou set $\hat{\bbC}\setminus J$, since an appropriate sequence of iterates of $F$ is bounded on $h^{-1}(U)$, hence, is normal. At the same time, we have the inclusion $h^{-1}(J(\cH))\subset J$, since by construction, the set $J(\cH)$ consist of Lyapunov unstable points. Thus, $h^{-1}(U)$ is a connected component of the Fatou set of $F$. According to Sullivan's No Wandering Domains Theorem, $h^{-1}(U)$ is not a wandering domain of $F$, hence $U$ is not a wandering domain of the shadow $S_\cH$. This leads to a contradiction.

The second assertion of the Lemma follows from combining non-existence of wandering domains with non-existence of periodic domains, as shown in Lemma~\ref{no_periodic_components_lemma}.
\end{proof}

\begin{proof}[Proof of Proposition~\ref{K_is_J_proposition}]
Assume that the filled Julia set $K(\cH)$ has interior points. Then let $U\subset K(\cH)\setminus J(\cH)$ be a nonempty connected component of $K(\cH)\setminus J(\cH)$. 

Let $\zeta$ be the underlying multicritical commuting pair of $\cH$. Then, according to Theorem~\ref{qc_conjugacy_theorem} and Remark~\ref{gluing_remark}, there exists a map $f\colon\bbT\to\bbT$ as in Lemma~\ref{model_family_lemma} and a positive integer $n\in\bbN$, such that $p\cR^nf$ and $p\cR^n\zeta$ both extend to holomorphic pairs $\cH_f$ and $\cH_\zeta$ respectively, and the pairs $\cH_f$ and $\cH_\zeta$ are quasiconformally conjugate. We may assume that $n$ is sufficiently large, so that the domain of $\cH_\zeta$ is contained in the domain of $\cH$. Let $(\tl\eta,\tl\xi)$ be the underlying commuting pair of $\cH_\zeta$.

Take an arbitrary point $z\in U$. According to Lemma~\ref{Fatou_convergence_to_R_lemma}, we have $S_\cH^k(z)\to [\eta(0),\xi(0)]$ as $k\to\infty$, hence, there exists $m\in\bbN$, such that the point $x:= S_\cH^m(z)$ is contained in $\Omega_{\cH_\zeta}$, and for all $k\in\bbN$, all further iterates $S_{\cH_\zeta}^k(x)$ are contained in $\Omega_{\cH_\zeta}$, and
\begin{equation}\label{no_wandering_convergence_condition} 
S_{\cH_\zeta}^k(x)\to [\tl\eta(0),\tl\xi(0)]\qquad\text{ as } k\to\infty.
\end{equation}

Observe that $x\not \in J(\cH)$. Since $J(\cH_\zeta)\subset J(\cH)$ (which follows, e.g., from the first assertion of Lemma~\ref{expansion_on_J_lemma}), we conclude that $x\not \in J(\cH_\zeta)$. On the other hand, since the two holomorphic pairs $\cH_\zeta$ and $\cH_f$ are quasiconformally conjugate, it follows from Lemma~\ref{no_wandering_lemma} that the orbit of the point $x$ under the dynamics of the shadow $S_{\cH_\zeta}$ escapes to the annulus $\Delta_{\cH_\zeta}\setminus \Omega_{\cH_\zeta}$. The latter contradicts the convergence~(\ref{no_wandering_convergence_condition}).
\end{proof}

\begin{remark}
The assumption that the rotation number of $\cH$ in Proposition~\ref{K_is_J_proposition} is of bounded type, was used only in the last step of the proof, where Theorem~\ref{qc_conjugacy_theorem} was applied. Proposition~\ref{K_is_J_proposition} will immediately hold for an arbitrary rotation number, once the corresponding generalization of Theorem~\ref{qc_conjugacy_theorem} is proved.
\end{remark}

\section{Proof of renormalization convergence in the analytic case}

\subsection{Critical points are deep}
For a point $z\in\bbC$ and a real number $r>0$, let $D(z,r)\subset\bbC$ denote the open disk of radius $r$, centered at $z$. We recall the following definition from~\cite{McM-ren2}:

\begin{definition}
Let $\Lambda\subset\bbC$ be a compact set and let $\delta>0$ be a positive real number. We say that $z\in\Lambda$ is a \textit{$\delta$-deep point} of $\Lambda$ if for every sufficiently small $r>0$, the largest disk, contained in $D(z,r)\setminus\Lambda$, has radius not greater than $r^{1+\delta}$.
\end{definition}



The following statement shows that $0$ is a $\delta$-deep point of a multicritical holomorphic pair (c.f.~\cite[Theorem 5.1]{FM2}).

\begin{proposition}\label{delta_deep_proposition}
Let $\cH$ be a multicritical holomorphic pair with an arbitrary irrational rotation number of bounded type. Then there exists $\delta>0$ such that the critical point at zero is a $\delta$-deep point of $J(\cH)$. The number $\delta$ depends on the total criticality $D$ of the pair $\zeta_{\cH}$ and the bounded type $B$ of its rotation number.
\end{proposition}
\begin{proof}
Let $m = m(B,D)$ be the same as in Theorem~\ref{ComplexBounds_theorem}. 
According to complex {\it a priori} bounds (Theorem~\ref{ComplexBounds_theorem} and Remark~\ref{gluing_remark}), there exists $n_0 = n_0(\cH)$, with the property that for every $n\ge n_0$ the prerenormalization $p\cR^n\zeta_{\cH}$ extends to a holomorphic pair $\cH_n$ with domain $\Omega_n$ and range $\Delta_n$, such that $\Delta_{n}\subset\Omega_\cH$, and an appropriate affine rescaling of $\cH_n$ is in $\mathbf H^D(m)$. Choose an increasing sequence of integers $n_0<n_1<n_2<\dots$, such that for each index $j\in\bbN$, we have $\Delta_{n_j}\subset\Omega_{n_{j-1}}$. It follows from Lemma~\ref{basic_holo_pair_lemma} that $\diam(\Delta_{n_j})$, the dynamical intervals of $\zeta_{\cH_{n_j}}$ and the radius of the largest disk, centered at zero and contained in $\Omega_{n_j}$, are all pairwise commensurable with an implicit constant depending on $m=m(B,D)$ and $D$. Hence, real {\it a priori} bounds (Theorem~\ref{real_bounds_thm}) imply that the numbers $n_j$ can be chosen so that $n_j\le n_0+ C_1j$, for all $j\in\bbN$, where $C_1$ is a constant that depends on $m=m(B,D)$ and $D$.

We will show that zero is a $\delta$-deep point of $J(\cH_{n_0})$. Since $J(\cH_{n_0})\subset J(\cH)$, this will be sufficient for the proof of Proposition~\ref{delta_deep_proposition}. In order to simplify the notation, we may assume that $n_0=0$, so $\cH = \cH_{n_0}$. Then we have
\begin{equation}\label{n_j_growth_eq}
n_j\le C_1j,\qquad \text{for all } j\in\bbN.
\end{equation}

Let $z\subset\Omega_\cH$ be an arbitrary point such that $z\not\in J(\cH)$. Let $k$ be the largest index, for which $z\in\Omega_{n_k}$. Then, due to Lemma~\ref{basic_holo_pair_lemma}, we have $|z|\asymp\diam(\Omega_{n_k})$ with the implicit constant depending on $B$ and $D$. If $\zeta_{\cH_{n_k}} = (\eta_k,\xi_k)$ and $I_{n_k}=[\eta_k(0),\,\xi_k(0)]$, then it follows from real {\it a priori} bounds (Theorem~\ref{real_bounds_thm}) that
$$
n_k\asymp -\log(|I_{n_k}|)\asymp -\log|z|
$$
with implicit constants depending on $B$ and $D$. Combining this with~(\ref{n_j_growth_eq}), we obtain
$$
k\ge -C_2\log|z|,
$$
where $C_2>0$ depends on $B$ and $D$.


Since $z\not\in J(\cH)$, its forward orbit under the dynamics of $S_\cH$ is finite and eventually leaves each domain $\Omega_{n_j}$, for $j=0,1,\dots, k$. Let $r\in\bbN$ be such that $S_\cH^r(z)$ is defined and not contained in $\Omega_\cH$. Then $S_\cH^r$ can be decomposed as follows:
$$
S_\cH^r = S_{\cH_{n_0}}^{r_0}\circ S_{\cH_{n_1}}^{r_1}\circ\ldots\circ S_{\cH_{n_{k-1}}}^{r_{k-1}}\circ S_{\cH_{n_k}}^{r_k},
$$
where the positive integers $r_0,\dots,r_k$ are determined so that if
$$
F_j := S_{\cH_{n_j}}^{r_j}\circ \ldots\circ S_{\cH_{n_k}}^{r_k},
$$
then $F_j(z)$ is the first point in the orbit of $z$ that leaves the domain $\Omega_{n_j}$. Now it follows from Lemma~\ref{expansion_at_escaping_points_lemma} and Lemma~\ref{basic_expansion_lemma} that for each $j=1,2,\dots, k$, we have
$$
\|F_j'(z)\|\ge \lambda^{k-j+1},
$$
where the norm is taken in the hyperbolic metric of $\Delta_{\cH}\setminus\bbR$, and $\lambda = \lambda(m(B,D), D)>1$ is the same as in Lemma~\ref{expansion_at_escaping_points_lemma}. In particular, it follows that
\begin{equation}\label{lambda_k_eq}
\|(S_\cH^{r-1})'(z)\|\ge \lambda^{k} \ge \lambda^{-C_2\log|z|} = |z|^{-C_2\log\lambda}.
\end{equation}

Next, we observe that according to Lemma~\ref{strip_lemma}, if $w=S_\cH^{r-1}(z)$, then $|\Im w|>\varepsilon |I_{n_0}|$, for some constant $\varepsilon=\varepsilon(B,D)>0$. The latter implies that the distance from $w$ to $J(\cH)$ in the hyperbolic metric of $\Delta_{\cH}\setminus\bbR$ is bounded by a constant $C_3=C_3(B,D)>0$. Let $\gamma\subset \Delta_{\cH}\setminus\bbR$ be the shortest geodesic arc, joining $w$ with $J(\cH)$. The appropriate inverse branch of $S_\cH^{r-1}$ takes $\gamma$ to a curve $S_\cH^{-(r-1)}(\gamma)$ that connects $z$ to some point of $J(\cH)$. Due to real symmetry, without loss of generality we may assume that $\gamma$ lies in the upper half-plane $\bbH$. Since its length is bounded above by $C_3$, it follows that the modulus of the annulus $(\Delta_\cH\cap\bbH)\setminus\gamma$ is bounded from below by a positive constant $C_4 = C_4(B,D)>0$. Then,  Koebe Distortion Theorem implies that
$$
\left\|\left(S_\cH^{-(r-1)}\right)'(\tl w)\right\| \asymp \left\|\left(S_\cH^{-(r-1)}\right)'(w)\right\|,
$$
for all points $\tl w\in\gamma$, where the implicit constant depends on $B$ and $D$. Combining this with~(\ref{lambda_k_eq}), we obtain that 
$$
\dist_{\Delta_\cH\setminus\bbR}(z, J(\cH))\le C_5|z|^{C_2\log\lambda}.
$$

Finally, comparing the hyperbolic and the Euclidean distances, it follows from a theorem of Beardon and Pommerenke~\cite[Theorem 1]{Beardon_Pommerenke_78} (see also~\cite[Theorem 2.3]{McM-ren1}) that 
$$
\dist(z, J(\cH))\le C_6|z|^{1+C_2\log\lambda},
$$
for some constant $C_6=C_6(B,D)>0$ and all $z\in\Omega_{\cH}\setminus J(\cH)$ sufficiently close to zero. This implies that zero is a $\delta$-deep point of $J(\cH)$, for any positive $\delta<C_2\log\lambda$.
\end{proof}

\begin{remark}
The bound $B$ on the combinatorial type was used in the above proof only for obtaining the modulus $m=m(B, D)$ from complex {\it a priori} bounds. Hence, it will follow immediately that $\delta$ does not depend on $B$, once complex bounds are established for arbitrary rotation numbers in the multicritical case.
\end{remark}

\subsection{Nonlinearity of the  dynamics on $J(\cH)$}
In this subsection we will show that the dynamics of a multicritical holomorphic pair is nonlinear in the sense of~\cite[Section 9]{McM-ren2}. We start by introducing the following definition:

\begin{definition}\label{lambda_nonlin_def}
For a real number $\lambda>1$, an analytic map $g\colon U\to V$ is called \textit{$\lambda$-nonlinear} if the following properties hold:
\begin{itemize}
	\item $g$ is a branched covering of $U$ onto $V$;
	\item there exist two points $u\in U$ and $v\in V$, such that $v=g(u)$ and $g'(u)=0$;
	\item if $A$ and $B$ are any two numbers from the collection of distances $\diam U$, $\diam V$, $\dist(u,\partial U)$, $\dist(v,\partial V)$, $\dist(u,v)$, then the following inequality holds:
	$$
	\frac{A}{\lambda}\le B\le\lambda A.
	$$
\end{itemize}
\end{definition}

The main goal of this subsection is to prove the following Proposition (c.f.,~\cite[Theorem 8.10]{McM-ren2}).

\begin{proposition}\label{small_J_everywhere_proposition}
For any real number $m\in(0,1)$ and a pair of positive integers $B>0$, $D>1$, there exists a real number $\alpha = \alpha(B, m, D)>1$, such that for any holomorphic pair $\cH = (\eta,\xi,\nu)\in \mathbf H^D(m)$ with an irrational rotation number of type bounded by $B$, any $r\in (0, 1)$ and any $z\in J(\cH)$, there exists a topological disk $O\subset\Omega_\cH$ that satisfies the following properties:
\begin{enumerate}
	\item $\cfrac{r}{\alpha} <\diam (O) < \alpha r$;
	\item $\dist(z, O)\le \alpha r$;
	\item there exist $n, k\in\bbN$, such that the composition $S_\cH^{-n}\circ S_\cH^{k+n}$ is well defined and is a $\alpha$-nonlinear map of degree $\le D^4$ on $O$, for an appropriate inverse branch $S_\cH^{-n}$.
\end{enumerate}
\end{proposition}

The proof will rely on several lemmas:

\begin{lemma}\label{lambda_nonlinear_compactness_lemma}
For any real $\lambda>1$ and a positive integer $D\ge 2$, the family of $\lambda$-nonlinear maps $g\colon U\to V$, satisfying
$$
\frac{1}{\lambda}\le \diam(U)\le\lambda, 
$$
and having topological degree not greater than $D$, is sequentially compact with respect to the Carath\'eodory convergence.
\end{lemma}
\begin{proof}
Given a sequence $g_n\colon (U_n, u_n)\to (V_n, v_n)$ of such $\lambda$-nonlinear maps, it follows from the last property of Definition~\ref{lambda_nonlin_def} that after extracting a subsequence, the sequences of pointed domains $(U_n,u_n)$ and $(V_n,v_n)$ Carath\'eodory converge to some pointed domains, as $n\to\infty$. Taking a further subsequence, we ensure that the maps converge as well. It is clear that all conditions of Definition~\ref{lambda_nonlin_def} are preserved after passing to a limit, which completes the proof.
\end{proof}

\begin{lemma}\label{lambda_nonlin_maps_lemma}
For any real number $m>0$ and a pair of integers $B>0$, $D>1$, there exists a real numbers $\lambda = \lambda(m, B, D)$, such that the following holds. Let $\cH = (\eta,\xi,\nu)\in\mathbf H^D(m)$ be a holomorphic pair with an irrational rotation number of type bounded by $B$. Then for any $r\in (0,1)$ and any $z\in \Omega_\cH\cap\bbR$, an appropriate composition of the maps $\eta$, $\xi$ and $\nu$ restricts to a real-symmetric map $h\colon U\to V$ with the following properties:
\begin{enumerate}
\item $z\in U$, the map $h$ satisfies $h=g\circ S_\cH^s$, where $s\in[0,1/m]$ and the map $g\colon S_\cH^s(U)\to V$ is $\lambda$-nonlinear of degree not greater than $D^4$;
\item if $z\in [\eta(0),\xi(0)]$, then $s=0$;
\item $r/\lambda\le\diam(U)\le \lambda r$.
\end{enumerate}
\end{lemma}



\begin{proof}
For a point $x\in [\eta(0),\xi(0)]\setminus\{0\}$, we define
$$
\zeta(x) =
\begin{cases}
\eta(x) &\text{if } x\in (0,\xi(0)] \\
\xi(x) &\text{if } x\in [\eta(0), 0).
\end{cases}
$$
Note that if the forward orbit of a point $x\in [\eta(0),\xi(0)]$ never lands at zero, then the iterates $\zeta^j(x)$ are defined for all $j\in\bbN$. Now, for any integer $j\ge 2$, let $\theta_j$ denote either the map $\eta$, if $\zeta^{j-2}(\eta(\xi(0)))\in (0,\xi(0))$, or the map $\xi$, if $\zeta^{j-2}(\eta(\xi(0)))\in (\xi(0), 0)$. For $j=0, 1$, define $\theta_0=\eta$ and $\theta_1 = \xi$. We will also identify the maps $\theta_j$ with their analytic extensions.

For all $n\ge 1$, define $\zeta_n = (\eta_n,\xi_n) = p\cR^n\zeta_{\cH}$, and consider open subintervals $I_n=(\eta_n(0), \, \eta_n^{-2}(\xi_n(0)))$ of the real line. (Analytic continuation of $\eta_n$ along the real line might be required in order to define the point $\eta_n^{-2}(\xi_n(0))$. Such analytic continuation always exists when $n\ge 1$.) Consider the sequence of intervals, defined inductively:
$$
I_n^0 = I_n, \quad\text{and}\quad I_n^j = \theta_{j-1}(I_n^{j-1}),\quad\text{for }j>0.
$$
Also define the interval
$$
\tl I_n^1 = \xi(I_n^0).
$$

Let $k_n$ be such that $I_n^{k_n} = \eta_n(I_n^0)$.
It follows easily that each point of the interval $[\eta(0),\xi(0)]$ belongs to at least one interval from the collection
$$
\mathcal I_n = \{I_n^0,\dots, I_n^{k_n}\}\bigcup \{\tl I_n^1\},
$$
and no point from $[\eta(0),\xi(0)]$ belongs to more than four distinct intervals from $\mathcal I_n$. 
Then real {\it a priori} bounds (Theorem~\ref{real_bounds_thm}) imply that there exists $n_0=n_0(m, D)\in\bbN$, such that for all $n\ge n_0$ and $\cH\in\mathbf H^D(m)$, the intervals $I_n^j$, $j\in[0,k_n]$, are sufficiently small so that due to property~(\ref{real_bound_property}) from Definition~\ref{Hm_def} and  Koebe Distortion Theorem, the following holds: 
if $V_n^{k_n}$ is the open round disk with diameter $I_n^{k_n}$, and the domains $V_n^j$, for $j\in[0,k_n-1]$, are defined inductively as connected components of $\theta_j^{-1}(V_n^{j+1})$, containing the corresponding intervals $I_n^j$, then $\diam(V_n^j) \asymp |I_n^j|$ with the implicit constant depending on $m$ and $D$. (The claim follows from using property~(\ref{real_bound_property}) of Definition~\ref{Hm_def} to express the inverse dynamics from a definite neighborhood of $V_n^{k_n}$ to a neighborhood of $V_n^0$ as a finite composition of univalent maps and roots. The number of maps in this composition is a function of $D$. When the interval $I_n^{k_n}$ is sufficiently short, then the claim follows from property~(\ref{real_bound_property}) of Definition~\ref{Hm_def} and  Koebe Distortion Theorem.)


Next, for each $j\in [0, k_n]$, let the point $v_j\in I_n^j$ be again defined inductively:
$$
v_0=0, \quad \text{and}\quad v_j = \theta_{j-1}(v_{j-1}),\quad\text{for }j>0.
$$
It follows from real {\it a priori} bounds (Theorem~\ref{real_bounds_thm})
that the points $v_j$ split the corresponding intervals $I_n^j$ into commensurable subintervals with the implicit constant depending on $m$ and $D$, hence, property~(\ref{real_bound_property}) from Definition~\ref{Hm_def} together with  Koebe Distortion Theorem imply that for all $n\ge n_0$ and each $j$, we have $\dist(v_j, V_n^j)\asymp \diam(V_n^j)$ with the implicit constant depending again on $m$ and $D$. 

Finally, pulling back each pointed domain $(V_n^j, v_j)$ by the $k_n$-th iterate $\zeta^{k_n}$ and applying 
 Koebe Distortion Theorem in exactly the same way as before, we arrive to a collection of real-symmetric pointed domains $(U_n^j, u_j)$, such that $\dist(u_j,v_j)$, $\diam U_n^j$, $\dist(u_j,\partial U_n^j)$ and $\diam V_n^j$ are all commensurable for the same value of $j$, and the $k_n$-th iterate of $\zeta$ is a $\lambda$-nonlinear map from $U_n^j$ to $V_n^j$, for some $\lambda>1$. The same construction can also be carried out for the domain $\tl V_n^1 = \xi(V_n^0)$ and its corresponding pullback $\tl U_n^1 = \xi(U_n^0)$.

It is clear that the parameter $\lambda$ so far depends only on $m$ and $D$. 
Since the union
$$
\left(\bigcup_{j=0}^{k_n} U_n^j\right)\bigcup \tl U_n^1
$$
contains a neighborhood of the interval $[\eta(0),\xi(0)]$, it follows that for any $z\in [\eta(0),\xi(0)]$ and any $n\ge n_0$, there exists an index $j$, such that $z\in U_n^j$. If $z\in(\Omega_\cH\cap\bbR)\setminus [\eta(0),\xi(0)]$, then, according to property~(\ref{k_property}) of Definition~\ref{Hm_def}, there exists $s\in[1,1/m]$, such that $S_\cH^s(z)\in [\eta(0),\xi(0)]$. Thus, for each $n\ge n_0$, we have constructed a map $h$ that satisfies the first two properties of Lemma~\ref{lambda_nonlin_maps_lemma}.

Now it follows from real {\it a priori} bounds (Theorem~\ref{real_bounds_thm}) that possibly after increasing $\lambda$, we can ensure that there exists $r_0 = r_0(n_0)>0$, such that for any $0<r<r_0$, one can find the level of renormalization $n\ge n_0$, for which the above constructed map $h\colon U\to V$ satisfies the last property of Lemma~\ref{lambda_nonlin_maps_lemma}.
The new value of $\lambda$ depends on $m$, $D$ and $B$. The number $r_0$ depends on $n_0$, which in turn depends only on $m$ and $D$. Replacing $\lambda$ by $\lambda/r_0$, we remove the requirement that $r<r_0$, which completes the proof of the lemma.

\end{proof}


\begin{lemma}\label{path_to_real_line_proposition}
For any real number $m\in(0,1)$ and a pair of positive integers $B>0$, $D>1$, there exists a real number $\beta = \beta(B, m, D)>1$, such that for any holomorphic pair $\cH = (\eta,\xi,\nu)\in \mathbf H^D(m)$ with an irrational rotation number of type bounded by $B$, any $r\in (0, 1)$ and any $z\in J(\cH)$, there exists a topological disk $O\subset\Omega_\cH$ that satisfies the following properties:
\begin{enumerate}
	\item $\cfrac{r}{\beta} <\diam (O) < \beta r$;
	\item $\dist(z, O)\le \beta r$;
	\item there exists $n\in\bbN$, such that $S_\cH^n$ is univalent on $O$ with distortion bounded by $\beta$, and $S_\cH^n(O)$ is a round disk with diameter contained in the interval $[\eta(0),\xi(0)]$.
\end{enumerate}
\end{lemma}
\begin{proof}
For each $j\ge 0$, define $z_j:= S_{\cH}^j(z)$. Let $v_0$ be a tangent vector to $\bbC$ at $z_0$ with $|v_0|=r$, and let $v_j$ be the tangent vector at $z_j$, satisfying $v_j = (S_{\cH}^j)'(z)v_0$. 
Let $\|v_j\|$ denote the length of the vector $v_j$ with respect to the hyperbolic metric on $\Delta_\cH\setminus \bbR$. According to~\cite[Lemma~4.6]{FM2}, there exists a constant $C_m$ that depends on $m$, such that
$$
\frac{|v_j|}{|\Im z_j|}\le \|v_j\| \le C_m\frac{|v_j|}{|\Im z_j|}.
$$
Let $\varepsilon>0$ be a sufficiently small number that will depend only on $m$ and $D$. 

According to Lemma~\ref{expansion_on_J_lemma}, either the orbit $\{z_j\}$ eventually lands on the real line, or $\|v_j\|\to\infty$ as $j\to\infty$. In the second case it follows from property~(\ref{Poincare_property}) of Definition~\ref{Hm_def} that $|\Im z_j|\asymp \dist(z_j, \overline{\Omega_\cH\cap\bbR})$ with the implicit constant depending on $m$. In both cases we conclude that there exists the smallest index $k\ge 0$, such that 
$$
\dist(z_k, \overline{\Omega_\cH\cap\bbR})\le|v_k|/\varepsilon.
$$
(We can always assume that $|v_{k}|>0$. Otherwise, if $|v_{k}|=0$, then $z_{k-1}$ is a critical point of the shadow $S_{\cH}$, and an arbitrarily small perturbation of $z_0$ eliminates this possibility.) 

First, we claim that $|v_k|<C_1$, for some positive constant $C_1>0$ that depends on $\varepsilon$, $m$ and $D$. Indeed, if $k=0$, then $|v_k|\le 1$. On the other hand, if $k>0$, then, according to the choice of $k$, we have 
\begin{equation}\label{v_k1_upper_bound_eq}
|v_{k-1}|<\varepsilon \dist(z_{k-1}, \overline{\Omega_\cH\cap\bbR})<\varepsilon/m,
\end{equation}
and since the derivatives of the shadows $S_\cH$ are uniformly bounded over $\cH\in\mathbf H^D(m)$, this implies that 
\begin{equation*}
|v_k|<C_2\varepsilon,
\end{equation*}
where $C_2>0$ is a constant that depends on $m$ and $D$, and the claim follows.

Let $\lambda = \lambda(B, D, m)$ be the same as in Lemma~\ref{lambda_nonlin_maps_lemma}. 
Then, according to Lemma~\ref{lambda_nonlin_maps_lemma}, an appropriate composition of the maps $\eta$, $\xi$ and $\nu$ restricts to a real symmetric map $h\colon U\to V$ that can be decomposed as
$$
h = g\circ S_\cH^s,
$$
where $s\in[0,1/m]$, the map $g\colon S_\cH^s(U)\to V$ is $\lambda$-nonlinear of degree not greater than $D^4$,
$$
\diam U\asymp |v_k|\qquad\text{and}\qquad \dist(z_{k}, U)<C_3 \diam(U),
$$
where the implicit constant and the constant $C_3>0$ depend on $m$, $B$, $D$ and $\varepsilon=\varepsilon(B, D, m)$.

From compactness considerations (c.f., Lemma~\ref{lambda_nonlinear_compactness_lemma} and Lemma~\ref{H_compactness_lemma}), one can always find a topological disk $W\subset U\setminus\bbR$, such that $h$ takes $W$ univalently onto a real-symmetric round disk, the distortion of $h|W$ is uniformly bounded, the identities $\diam W\asymp\diam U\asymp \dist(W,\bbR)$ hold with implicit constants depending on the parameters $\lambda = \lambda(B, D, m)$, $m$, $D$, and the point $z_{k}$ lies in the same half-plane with the domain $W$. 
Finally, define the disk $O$ as the (univalent) pullback of $W$ by the branch of $S_\cH^{-k}$, taking $z_{k}$ back to $z_0$. 
Observe that if $k=0$, then the proof is complete, so from now on we assume that $k>0$.

Now, there are two possibilities: if 
$|\Im z_{k}|>\varepsilon|v_{k}|$,
then  Koebe Distortion Theorem implies that the disk $O$ satisfies the conditions of the Lemma~\ref{path_to_real_line_proposition} for some constant $\beta$ that depends on $B$, $D$ and $m$, and the proof is complete.

If $|\Im z_{k}|\le\varepsilon|v_{k}|$, then according to~(\ref{v_k1_upper_bound_eq}), the constant $\varepsilon<1$ can be chosen sufficiently small to guarantee that $|v_{k-1}|<m^2/2$. Then, according to property~(\ref{extension_property}) of Definition~\ref{Hm_def}, the shadow $S_\cH$ has an extension to the round disk $D(z_{k-1}, 2|v_{k-1}|)$ as an analytic map with topological degree not greater than $1/m$. Define $\tl D$ as the disk $\tl D := D(z_{k-1}, |v_{k-1}|)$. Then the Koebe Distortion Theorem implies that the domain $S_\cH(\tl D)$ contains the disk $D(z_k,C_4|v_k|)$, for some constant $C_4>0$ that depends on $m$. Now, if the parameter $\varepsilon$ is sufficiently small, then the condition $|\Im z_{k}|\le\varepsilon|v_{k}|$ implies that the domain $S_\cH(\tl D)$ contains a (relatively) large part of the interval $\Omega_\cH\cap\bbR$, hence, according to Lemma~\ref{lambda_nonlin_maps_lemma}, one can choose the domain $U$ so that $U\subset S_\cH(\tl D)$. Then, the (multivalent) Koebe Theorem implies that if $W'$ is the univalent pullback of $W$ by the branch of $S_\cH^{-1}$, taking $z_{k}$ to $z_{k-1}$, then 
$$
\diam W'\asymp |v_{k-1}|\qquad\text{and}\qquad \dist(z_{k-1}, W')<C_5 \diam(W'),
$$
where the implicit constant and the constant $C_5>0$ depend on $m$, $B$, $D$ and $\varepsilon=\varepsilon(B, D, m)$.

Finally, observe that according to~(\ref{v_k1_upper_bound_eq}), we have $$\dist(W'\cup\{z_{k-1}\},\bbR)\asymp |v_{k-1}|,$$ hence, the proof of Lemma~\ref{path_to_real_line_proposition} is completed by applying the (univalent) Koebe Distortion Theorem to the inverse branch of $S_\cH^{-k+1}$, taking $z_{k-1}$ back to $z_0$.

\end{proof}

\begin{proof}[Proof of Proposition~\ref{small_J_everywhere_proposition}]
The proof is obtained immediately by combining Lemma~\ref{lambda_nonlin_maps_lemma} and Lemma~\ref{path_to_real_line_proposition}. More specifically, Lemma~\ref{path_to_real_line_proposition} implies that some iterate $S_\cH^n$ takes a neighborhood $O'$ univalently onto a real-symmetric round disk with diameter contained in the interval $[\eta(0),\xi(0)]$. Furthermore, the neighborhood $O'$ satisfies the first two conditions of Proposition~\ref{small_J_everywhere_proposition}, for an appropriate real number $\alpha$, where $\alpha$ is independent of $r$ and $z$. Next, according to Lemma~\ref{lambda_nonlin_maps_lemma}, there exists $k\in\bbN$, such that $S_\cH^k$ restricts to a $\lambda$-nonlinear map $g\colon U\to V$ of degree not greater than $D^4$, with the domain and range both contained in $S_\cH^k(O')$ and having diameters commensurable with $\diam(S_\cH^k(O'))$. Finally, we define $O:= S_\cH^{-n}(U)$. Since the inverse branch $S_\cH^{-n}$ has bounded distortion, then, possibly, after increasing the constant $\alpha$, one ensures that all conditions of Proposition~\ref{small_J_everywhere_proposition} are satisfied.
\end{proof}

\subsection{Geometric convergence of renormalizations and $C^{1+\alpha}$-rigidity for analytic maps}
\label{subsec-proofs-analytic}
The proof of  Theorem~\ref{th:rencov2} will now follow by a straightforward application of C.~McMullen's general result on dynamical inflexibility~\cite[Theorem 9.15]{McM-ren2}. We will restate his theorem in a convenient form for holomorphic pairs.

A line field $\mu$ on $\hat \bbC$ is \textit{parabolic} if $\mu = A^*(dz/d\overline z)$ where $A\in\mathrm{Aut}(\hat\bbC)$. A holomorphic map $f\colon U\to V$ between two domains $U, V\subset\hat\bbC$ is \textit{nonlinear} if it does not preserve a parabolic line field. The nonlinearity $\nu(f)$ is defined as follows:
$$
\nu(f):= \inf_{\mathrm{parabolic} \,\,\mu} \left(\sup_{B\subset U} \int_B |\mu-f^*\mu| \sigma^2(z) dx\,dy \right),
$$
where $\sigma(z)$ denotes the density of the spherical metric, and the supremum is taken over all round disks $B$, such that $f$ is defined and univalent on the round disk of twice the radius. It is straightforward that $\nu(f)\neq 0$ if and only if $f$ is nonlinear.

\begin{definition}
Given $\varepsilon>0$, a holomorphic pair $\cH$ is \textit{$\varepsilon$-uniformly nonlinear} (with respect to its Julia set $J(\cH)$), if for any $z\in J(\cH)$ and any round disk $D\subset\Omega_\cH$ centered at $z$, there exist a domain $O\subset D$ and a single valued map $g = S_\cH^{-n}\circ S_\cH^{n+k}\colon O\to\bbC$ with the following property: 
if $A\in\mathrm{Aut}(\hat\bbC)$ is an affine rescaling of $D$ onto the unit disk, then the map $h = A\circ g\circ A^{-1}$ is nonlinear on $A(O)$ with $\nu(h)\ge \varepsilon$.
\end{definition}


\begin{definition}
A holomorphic pair $\cH$ is \textit{uniformly twisting} (with respect to its Julia set $J(\cH)$), if any holomorphic pair $\cH'$ quasiconformally conjugate to $\cH$ is $\varepsilon$-uniformly nonlinear, for some $\varepsilon>0$.
\end{definition}

For $K>1$, let $\nu^K(\cH)$ be the supremum of all values of $\varepsilon$, such that any $\cH'$,  $K$-quasiconformally conjugate to $\cH$, is $\varepsilon$-uniformly nonlinear. By compactness, it follows that if $\cH$ is uniformly twisting, then $\nu^K(\cH)>0$ for any $K$

\begin{theorem}\cite{McM-ren2}\label{McM_theorem}
Let $\cH$ be a uniformly twisting holomorphic pair and let $\phi\colon\hat\bbC\to\hat\bbC$ be a $K$-quasiconformal conjugacy that takes $\cH$ to another holomorphic pair $\cH'$. Then for any $\delta$-deep point $z$ of $J(\cH)$, the map $\phi$ is $C^{1+\beta}$-conformal at $z$. The constant $\beta>0$ depends on $K$, $\delta$ and $\nu^K(\cH)$.
\end{theorem}

\begin{proof}[Proof of Theorem~\ref{th:rencov2}]
Let $f$ and $g$ be two analytic multicritical circle maps satisfying all conditions of the theorem. Then, according to Theorem (cite), there exist $n\in\bbN$ and $K=K(B,D)\ge 1$ such that the renormalizations $\cR^nf$ and $\cR^ng$ extend to two $K$-quasiconformally conjugate holomorphic commuting pairs $\cH$ and $\cH'$ from $\mathbf H^D(m)$, where $m = m(B,D)>0$. 

It follows from Proposition~\ref{delta_deep_proposition} that zero is a $\delta$-deep point of $J(\cH)$, where $\delta = \delta(B,D)>0$. Proposition~\ref{small_J_everywhere_proposition} together with compactness of $\lambda$-nonlinear maps (Lemma~\ref{lambda_nonlinear_compactness_lemma}) imply that $\cH$ is uniformly twisting with $\nu^K(\cH)$ depending only on $B$ and $D$. Then it follows from Theorem~\ref{McM_theorem} that the holomorphic pairs $\cH$ and $\cH'$ are $C^{1+\beta}$-conformally conjugate at zero. This immediately implies geometric convergence of renormalizations.
\end{proof}

Finally, Theorem~\ref{thm-rigidity-analytic} follows from Theorem~\ref{th:rencov2} and the following 
reformulated version of Proposition~4.3 from~\cite{FM1}:
\begin{proposition}\label{dFdM43_prop}
Let $f$ be a $C^3$-smooth multicritical circle map with an irrational rotation number of type bounded by $B$, and let $h\colon\bbR\slash\bbZ\to\bbR\slash\bbZ$ be a homeomorphism. If there exist constants $C>0$ and $0<\lambda<1$ such that
\begin{equation}\label{dFdM_condition}
\left|\frac{|I|}{|J|}-\frac{|h(I)|}{|h(J)|}\right|\le C\lambda^n,
\end{equation}
for each pair of adjacent atoms $I,J\in\cI_n(f)$, for all $n\ge 0$, then $h$ is a $C^{1+\beta}$-diffeomorphism for some $\beta>0$ that depends only on $\lambda$ and $B$.
\end{proposition}

%% file: proof.tex
\section{Proof of Theorems \ref{th:rencov1} and \ref{main_rigidity_theorem}}\label{sec:reduction}
\subsection{Lipschitz property of renormalization}

Let us begin by quoting a Lipschitz property of renormalization operator (cf. Lemma A.2 of \cite{GdM}):
\begin{lemma}
  \label{lem:lip}
  Given $B\in\NN$, $D\in\NN$, and $C>0$ there exists $L>1$ such that the following is true. Let $\zeta_1$, $\zeta_2$ be two $C^3$-smooth multicritical commuting pairs
  whose total criticality is bounded by $D$ and whose heights satisfy $\chi(\zeta_1)=\chi(\zeta_2)\leq B$. Assume further, that the geometry of the first dynamical partition of at least one of the pairs is bounded by $C$, while the first dynamical partition of the other pair satisfies conditions~(1) and~(2) of Definition~\ref{bound_geom_def}. Then
   $$\dist_{C^0}(\cR\zeta_1,\cR\zeta_2)<L\dist_{C^0}(\zeta_1,\zeta_2).$$
\end{lemma}
We note that a Lipschitz property for unbounded type is far from straightforward to prove (cf. Lemma 4.1 of \cite{GMdM}).

\begin{remark}\label{Lip_remark}
The statement of Lemma~\ref{lem:lip} holds for $C^3$ multicritical circle maps instead of the commuting pairs. The proof is identical.
\end{remark}

For a bi-critical commuting pair, we will say that the critical point at zero is the \textit{marked} critical point.

\begin{lemma}\label{combinatorics_vs_distance_lemma}
For any $D\in\bbN$ and $C>0$, there exists $\varepsilon = \varepsilon(D,C)>0$ such that the following holds. Let $\zeta_1$ and $\zeta_2$ be two bi-critical or unicritical commuting pairs of class $C^3$ with the same irrational rotation number, the total criticalities bounded by $D$, and the geometries of all dynamical partitions bounded by $C$. If $\zeta_1$ and $\zeta_2$ have different 
signatures, then there exists a non-negative integer $n\ge 0$, such that
$$
\dist_{C^0}(\cR^n\zeta_1, \cR^n\zeta_2)>\varepsilon.
$$
\end{lemma}

\begin{proof}
Assume, the pairs $\zeta_1$ and $\zeta_2$ satisfy the conditions of the lemma. Then, for each $n\ge 0$ and $j=1,2$, let $A_{j,n}$ be the M\"obius transformations, such that $A_{j,n}(0)=0$, and the commuting pairs 
$$
\zeta_{j,n}:= A_{j,n}\circ\cR^n\zeta_j\circ A_{j,n}^{-1}
$$
are defined on the interval $[-1,1]$; that is, if $\zeta_{j,n}=(\eta,\xi)$, then $\eta(0)=-1$ and $\xi(0)=1$. For each $n\ge 0$ and $j=1,2$, let $c_{j,n}\in[-1,1]$ be the non-marked critical point of $\zeta_{j,n}$. If the pair $\zeta_{j,n}$ happens to be unicritical, then we set $c_{j,n}=0$.

The signatures of $\zeta_1$ and $\zeta_2$ can differ in two ways: either the combinatorial distances between the critical points are different, or they are the same, but the critical exponents of the critical points do not match. We start by considering the first possibility.

\underline{Claim 1:} If the combinatorial distance between the two critical points of $\zeta_1$ is different from the combinatorial distance between the two critical points of $\zeta_2$, then there exists $\varepsilon_1 = \varepsilon_1(C)>0$ and $n\ge 0$, such that one of the non-marked critical points $c_{1,n}$ and $c_{2,n}$ is at least distance $\varepsilon_1$ away from the other non-marked critical point and from each of the points $-1$, $0$, $1$. 

\underline{Proof of Claim 1:} 
First, we observe that there exists a sufficiently small constant $\varepsilon_2 = \varepsilon_2(C)>0$, such that 
if for some $k\ge 0$, we have $|c_{1,k}|, |c_{2,k}|<\varepsilon_2$, and the points $c_{1,k}$ and $c_{2,k}$ either lie on different sides from zero, or exactly one of them is equal to zero, then the claim follows almost immediately. Indeed, in this case, then boundedness of the geometry implies that after an appropriate number of renormalizations, the points will remain on different sides of zero, and at least one of them will be some fixed distance $\varepsilon_3$ away from $0$, $1$ and $-1$. This fixed distance $\varepsilon_3>0$ depends only on the constant~$C$.

In general, since $\rho(\zeta_1)=\rho(\zeta_2)$, the dynamical partitions of $\zeta_1$ and $\zeta_2$ are combinatorially the same at any level. 
If the combinatorial distances between two critical points of $\zeta_1$ and $\zeta_2$ are different, then there exists the smallest $k\in\bbN$, such that the non-marked critical points of $\zeta_1$ and $\zeta_2$ belong to combinatorially different intervals of the dynamical partitions of level $k$. This implies that $c_{1,k-3}$ and $c_{2,k-3}$ belong to combinatorially different intervals of the dynamical partitions of level $3$ for the pairs $\zeta_{1,k-3}$ and $\zeta_{2,k-3}$ respectively. At the same time, we can assume that $|c_{1,k-3}-c_{2,k-3}|<1-2\varepsilon_4$, for some constant $\varepsilon_4 = \varepsilon_4(C)>0$, since the points $c_{1,k-3}$ and $c_{2,k-3}$ belong to combinatorially the same intervals from the corresponding partitions of level $2$. This implies that if $|c_{1,k-3}-c_{2,k-3}|\ge\varepsilon_4$, then one of the non-marked critical points $c_{1,k-2}$ and $c_{2,k-2}$ is at least distance $\varepsilon_4$ away from the other non-marked critical point and from each of the points $-1$, $0$, $1$.

Finally, without loss of generality, we may assume that $\varepsilon_4$ is sufficiently small, so that if $|c_{1,k-3}-c_{2,k-3}|<\varepsilon_4$, then the points $c_{1,k+2}$ and $c_{2,k+2}$ satisfy the conditions, considered in the beginning of the proof. Namely, they lie on different sides from zero (or exactly one of them is equal to zero), and $|c_{1,k+2}|, |c_{2,k+2}|<\varepsilon_2$. Hence, we conclude that Claim~1 holds for
$$
\varepsilon_1 = \min\{\varepsilon_3, \varepsilon_4\}.
$$
This completes the proof of Claim~1.

\vspace{2mm}

Now let $\zeta_1$, $\zeta_2$, $\varepsilon_1$ and $n$ be the same as in the statement of Claim~1. Without loss of generality, we may assume that $c_{1,n}$ is the critical point that is at least distance $\varepsilon_1$ away from the other critical points of both pairs $\zeta_{1,n}$ and $\zeta_{2,n}$, as well as from the endpoints $-1$ and $1$. For any sufficiently small $\delta>0$, let $I\subset [-1,1]$ be the largest interval, centered at $c_{1,n}$, such that 
\begin{equation}\label{I_eq}
\zeta_{1,n}'(x)\le\delta,\qquad\text{for all } x\in I.
\end{equation}
Due to boundedness of the geometry, if $d\le D$ is the critical exponent of $c_{1,n}$, then 
\begin{equation}\label{I_commens_delta_eq}
|I|\asymp \delta^{\frac{1}{d-1}}
\end{equation}
with the implicit constant depending only on $C$ and on $d$.

At the same time, due to boundedness of the geometry, there exists $\varepsilon_5 = \varepsilon_5(\delta, C, D)>0$, such that if $x\in[-1,1]$ is at least distance $\varepsilon_5$ away from $-1$, $1$, and all critical points of $\zeta_{2,n}$, then $\zeta_{2,n}'(x)>2\delta$. Furthermore, $\varepsilon_5\to 0$ as $\delta\to 0$.

Fix $\delta>0$ so that $|I|<\varepsilon_1$ and $\varepsilon_5<\varepsilon_1/3$. Such $\delta$ depends only on $C$ and $D$. Then, according to Claim~1, we have $\zeta_{2,n}'(x)>2\delta$, for any $x\in I$. Together with~(\ref{I_eq}) and~(\ref{I_commens_delta_eq}), this implies that $\| \zeta_{1,n}-\zeta_{2,n}\|_{C^0(I)} \ge \varepsilon$, for some $\varepsilon>0$ that depends only on $\delta = \delta(D, C)$. Hence, we complete the proof of Lemma~\ref{combinatorics_vs_distance_lemma} in the case when the combinatorial distances between the critical points are different for the two pairs $\zeta_1$ and $\zeta_2$.

Next, suppose that the combinatorial distances between the critical points are the same for the two pairs $\zeta_1$ and $\zeta_2$, but the orders of the critical points of the pairs do not match. We will consider the case, when $\zeta_1$ and $\zeta_2$ are bi-critical; for unicritical pairs, an obvious modification of the same argument will work. Assume that the statement of Lemma~\ref{combinatorics_vs_distance_lemma} does not hold in this case. Then for any $k\in\bbN$, there exist two bi-critical pairs $\zeta_k$ and $\hat\zeta_k$ with different pairs of critical exponents, such that
\begin{equation}\label{false_comparison_eq}
\dist_{C^0}(\cR^n\zeta_k, \cR^n\hat\zeta_k)<\frac{1}{k},
\end{equation}
for any $n$.
Proceeding in the same way as in the beginning of the proof of Claim~1, we can ensure that for some $n\ge 0$ and some $\varepsilon_6 = \varepsilon_6(C)>0$, both non-marked critical points of $\cR^n\zeta_k$ and $\cR^n\hat\zeta_k$ are at least distance $\varepsilon_6$ away from $0$ and the endpoints of the dynamical intervals. 
Due to $C^0$ compactness of the class of pairs, whose dynamical partitions have the geometry bounded by $C$ at all levels, we can pass to a subsequence of tuples $(n, k)$, so that the corresponding subsequences of commuting pairs $\{\cR^n\zeta_k\}$ and $\{\cR^n\hat\zeta_k\}$ converge to some commuting pairs $\zeta$ and $\hat\zeta$ respectively, and the orders of the critical points within each subsequence are the same. By construction, both limiting pairs are bi-critical and have non-matching tuples of criticalities. On the other hand, condition~(\ref{false_comparison_eq}) implies that $\zeta = \hat\zeta$, hence their criticalities should match, which is a contradiction. This completes the proof of Lemma~\ref{combinatorics_vs_distance_lemma}.

\end{proof}

\subsection{Almost commuting pairs}
\begin{definition}
An \textit{almost commuting pair} of order $k\geq 3$ is a $C^k$ multicritical commuting pair $\zeta=(\eta,\xi)$, such that 
\begin{enumerate}
	\item the maps $\eta$ and $\xi$ have analytic extensions to some neighborhoods of their respective dynamical intervals $[\xi(0),0]$ and $[0,\eta(0)]$, (we will use the same symbols $\eta$ and $\xi$ to denote these analytic extensions);
	\item the maps $\eta\circ\xi$ and $\xi\circ\eta$ have criticality exactly $k$ at zero.
\end{enumerate}
\end{definition}

Let $[\zeta]= [\eta,\xi]$ denote the commutator
$$
[\zeta]\equiv \eta\circ \xi-\xi\circ\eta,
$$
defined in some complex neighborhood of the origin. 

We note that the commutation condition of appropriate $C^k$-smooth extensions of the maps $\eta|_{[\xi(0),0]}$ and $\xi|_{[0,\eta(0)]}$ 
in the setting of almost commuting pairs is equivalent to the following:
$$
[\zeta](x)=\eta\circ \xi(x)-\xi\circ\eta(x)=o(x^k) \qquad\text{as }x\to 0.
$$



Since the order of the critical point of $\eta\circ\xi$ at zero can only increase after renormalization, we evidently have
\begin{proposition}
If $\zeta$ is an almost commuting pair of order $k$, then its renormalization $\cR\zeta$ is also an almost commuting pair of order $\tl k\ge k$.

\end{proposition}
We will say that $\zeta$ is an almost commuting pair (without specifying its order), if it is an almost commuting pair of some order $k\ge 3$.

For a compact set $K\subset\bbC$ and $\eps>0$, let $U_\eps(K)$ be the $\eps$-neighborhood of $K$. That is,
$$
U_\eps(K) = \{z\in\bbC\mid \min_{w\in K}|z-w|<\eps\}.
$$
If $\zeta=(\eta,\xi)$ is an analytic multicritical commuting pair or an almost commuting pair, and $\eps>0$, we will say that $\zeta$ {\it has an  $\eps$-extension}, if
$\eta$ and $\xi$ have analytic extensions to $U_\eps([0,\xi(0)])$,
$U_\eps([0,\eta(0)])$ respectively, extending continuously to the boundaries of the domains and having all critical values on the real line.
In particular, one can introduce the sup-norm for such maps on the domains of their $\eps$-extensions, with the notations:
$$
\|\xi\|_\eps^\infty,\;\|\eta\|_\eps^\infty,\text{ and }\|\zeta\|_\eps^\infty := \max(\|\xi\|_\eps^\infty, \|\eta\|_\eps^\infty).
$$

\begin{definition}
For an integer $D\ge 3$ and real numbers $\eps>0$, $M>1$, we denote by $\mathcal C(\eps, M, D)$ the family of all almost commuting pairs $\zeta = (\eta,\xi)$ with $\xi(0)=1$, such that 
\begin{itemize}
	\item $\zeta$ has an $\eps$-extension with $\|\zeta\|_\eps^\infty \le M$,
	\item the sum of the orders of all critical points of $\eta$ and $\xi$ in $U_\eps([0,\xi(0)])$ and $U_\eps([0,\eta(0)])$ respectively is bounded by $2D$,
	\item the rotation number $\rho(\zeta)$ is irrational, and the geometry of the dynamical partitions of $\zeta$ at all levels is bounded by $M$.
\end{itemize}
\end{definition}

We start by the following basic observation:
\begin{lemma}\label{CeMD_compactness_lemma}
For any integers $r\ge 0$, $D\ge 3$ and real numbers $\eps>0$, $M>1$, the family $\mathcal C(\eps,M,D)$ is sequentially compact in the space of all almost commuting pairs with respect to the $C^r$-metric $\dist_{C^r}$.
\end{lemma}
\begin{proof}
Boundedness of the geometry implies that any sequence of pairs from $\mathcal C(\eps,M,D)$ has a subsequence $\{(\xi_n,\eta_n)\}_{n=1}^\infty$, such that all maps $\xi_j$ are analytic in some neighborhood $W\subset\bbC$ and all maps $\eta_j$ are analytic in some neighborhood $V\subset\bbC$ with the additional condition that $U_{\eps/2}([0,\xi(0)])\subset V$ and $U_{\eps/2}([0,\eta(0)])\subset W$. Then, boundedness of the sup-norm of the maps from $\mathcal C(\eps,M,D)$ implies that the sequence $\{(\xi_n,\eta_n)\}_{n=1}^\infty$ is normal. Thus, one can extract a further subsequence that converges to an almost commuting pair $\zeta$. It follows immediately that $\zeta$ satisfies all conditions of the above definition, hence, $\zeta\in \mathcal C(\eps,M,D)$.
\end{proof}


Next, we note that the proof of complex {\it a priori} bounds \cite{Est_Smania_Yampolsky_2020} (see also Theorem~\ref{ComplexBounds_theorem}) carries over {\it verbatim} to the case of almost commuting pairs. Together with Lemma~\ref{basic_holo_pair_lemma} and the compactness result of Lemma~\ref{CeMD_compactness_lemma}, this yields the following:
\begin{proposition}
	\label{prop-bounds1}
	There exist $\eps,M>0$, depending only on $B$ and $D$, such that the following holds. Let $\zeta$ be an almost commuting pair with an irrational rotation number of type bounded by $B$ and with total criticality bounded by $D$. Denote
	$$\zeta_n=(\eta_n,\xi_n)=\cR^n\zeta.$$
	Then there exists $n_0\in\bbN$, such that for all $n\ge n_0$, we have $\zeta_n\in\mathcal C(\eps,M,D)$.
	The constant $n_0$ can be chosen independently from the pair $\zeta$, provided that $\zeta$ belongs to some fixed family $\mathcal C(\tl\eps,\tl M,D)$.
\end{proposition}

The main goal of this subsection is to prove the following theorem:

\begin{theorem}
	\label{thm-near-comm2}
	For any $B\ge 1$ and $D\ge 3$, there exist $\lambda_1\in(0,1)$ and a compact (with respect to $C^r$-metric, for any $r\ge 0$) collection $\cC_{B,D}$ of analytic multicritical commuting pairs such that the following holds.
	For 
	any almost commuting pair $\zeta$ with the rotation number bounded by $B$ and total criticality not greater than $D$, there exist $M>0$ and a sequence of analytic commuting pairs $\{\zeta_n\}_{n=1}^\infty\subset\cC_{B,D}$, such that the inequality 
	$$
	\dist_{C^0}(\zeta_n, \cR^n\zeta)\le M\lambda_1^n
	$$
	holds for any $n\ge 1$, and the pairs $\cR^n\zeta$ and $\zeta_n$ have the same rotation number and the same number of critical points. Furthermore, the constant $M$ can be chosen independently from the pair $\zeta$, provided that $\zeta$ belongs to some fixed family $\mathcal C(\tl\eps,\tl M,D)$.
\end{theorem}





For a holomorphic map $f\colon U\to\bbC$ defined on an open domain $U\subset\bbC$, let $\|f\|_U^\infty$ denote the sup-norm of $f$:
$$
\|f\|_U^\infty:= \sup_{z\in U} |f(z)|.
$$

Our key observation is the following:
\begin{theorem}
  \label{thm-comm}
Let $\zeta$ be an almost commuting pair with an irrational rotation number. Then there exist $\tau\in(0,1)$, $m\in\NN$ and a neighborhood $U(0)\subset\bbC$ of zero, such that the following holds:  
  \begin{equation}
    \label{commutator-estimate}
    \|[\cR^m\zeta]\|^\infty_{U(0)}<\tau\|[\zeta]\|^\infty_{U(0)}.
  \end{equation}
The constants $\tau, m$ and the neighborhood $U(0)$ can be chosen independently from the pair $\zeta$, provided that $\zeta$ belongs to some fixed family $\mathcal C(\eps,M,D)$. 


\end{theorem}

\noindent
Before proving the theorem we will need the following lemma:
\begin{lemma}
\label{lem-comm}
Let $\zeta=(\eta,\xi)$ be an almost commuting pair of order $k\ge 3$ and an irrational rotation number. For any $n\in\bbN$, define $(\eta_n,\xi_n):=p\cR^n\zeta$. Then for any $n\in\bbN$, there exists an analytic map $f_n$ defined in a neighborhood of the point $\eta(\xi(0))$, such that
$$
\eta_n\circ\xi_n = f_n\circ\eta\circ\xi\qquad\text{and}\qquad \xi_n\circ\eta_n = f_n\circ\xi\circ\eta.
$$
In particular, the commutator $[p\cR^n\zeta]$ has the form
\begin{equation}
  \label{eq-comm}
  [p\cR^n\zeta]=f_n\circ\eta\circ\xi-f_n\circ\xi\circ\eta.
\end{equation}

Set $\lambda_n=|\xi_n(0)|$. 
Then, for all large enough values of $n$, we have
$$f_n'(\eta\circ\xi(0))= O(\lambda_n^{1-k}).$$
\end{lemma}
\begin{proof}
  The proof is easily supplied by induction. Indeed, assume that
  $$[p\cR^n\zeta]=f_n\circ \eta\circ \xi-f_n\circ \xi\circ \eta.$$
  Then
  $$[p\cR^{n+1}\zeta]=\eta_n^r\circ \eta_n\circ \xi_n-\eta_n^r\circ \xi_n\circ \eta_n=\eta_n^r\circ f_n\circ \eta\circ\xi-\eta_n^r\circ f_n\circ \xi\circ\eta,$$
  so $f_{n+1}=\eta_n^r\circ f_n$.

  The proof of the second statement goes as follows: by real {\it a priori} bounds, there exists a neighborhood $U\subset\bbR$ of zero, such that the map $\eta\circ\xi$ has no critical points in $U\setminus\{0\}$, and
  $$
  \sup_{n\in\bbN, x\in U\setminus\{0\}}\left(\frac{d}{dx}\left[\lambda_n^{-1}(\eta_n\circ\xi_n)(\lambda_n x)\right] /\,(\eta\circ\xi)'(x) \right)<\infty. 
  $$
  At the same time, we observe that 
  $$
  \frac{d}{dx}\left[\lambda_n^{-1}(\eta_n\circ\xi_n)(\lambda_n x)\right]= 
  f_n'(\eta\circ\xi(\lambda_nx))\cdot(\eta\circ\xi)'(\lambda_nx)
  $$
  and
  $$
  (\eta\circ\xi)'(x) \sim x^{k-1}\qquad\text{as }x\to 0.
  $$
  Combining all these observations and taking $x\to 0$, completes the proof.
\end{proof}

\begin{proof}[Proof of Theorem~\ref{thm-comm}]
  As before, for $n\in\bbN$, let $p\cR^n\zeta=(\eta_n,\xi_n)$ and set $\lambda_n=|\xi_n(0)|$.
  In view of Lemma~\ref{lem-comm}, we have the following first-order estimate around $x=0$:
  $$[p\cR^n\zeta]=\eta_n\circ\xi_n-\xi_n\circ\eta_n\sim f_n'(\eta\circ\xi(0))[\zeta]\sim f_n'(\eta\circ\xi(0))\cdot cx^{\hat k},$$
  for some constants $c>0$ and $\hat k>k$, such that $[\zeta](x) \sim cx^{\hat k}$.
  Thus, after rescaling we get
  $$|[\cR^n\zeta](x)|\sim \lambda_n^{\hat k-1} |f_n'(\eta\circ\xi(0))|\cdot |c||x|^{\hat k} =  O(\lambda_n^{\hat k-k})\cdot|c||x|^{\hat k} = O(\lambda_n^{\hat k-k})\cdot |[\zeta](x)|,$$
  hence, by real {\it a priori} bounds, for a sufficiently large $m\in\bbN$, there exist $\tau\in(0,1)$ and a neighborhood $U(0)$ that satisfy the statement of the theorem.

The second statement of the theorem follows immediately from sequential compactness of the family $\mathcal C(\eps,M,D)$.
\end{proof}

\noindent
An implication of Theorem~\ref{thm-comm} is that renormalizations of almost commuting pairs with rotation numbers of bounded type converge to analytic commuting pairs. Indeed, by Proposition~\ref{prop-bounds1}, the sequence of such renormalizations is relatively compact, and according to Theorem~\ref{thm-comm}, every limit point of this sequence is a multicritical commuting pair. However, exponential convergence of renormalizations of almost commuting pairs to the commuting ones in $C^r$-metric is not immediately obvious, and requires a proof, which we supply below.

For an  almost commuting pair $\zeta=(\eta,\xi)$ and a regular point $x\in[\eta(0),0]$ of the map $\xi$,
let us perform a gluing procedure from \S\ref{sec:glue} on the interval $[x,\xi(x)]$, identifying its ends by the map $\xi$. Since $\xi$ is locally conformal there, the resulting one dimensional manifold $S$ is an analytic circle. Moreover, the gluing extends to a neighborhood of the endpoints, pasting together the ends of an $\eps$-neighborhood of $[x,\xi(x)]$ to produce a Riemann surface $C\supset S$ whose topological type is a cylinder. Let us select a conformal identification of $C$ with a subannulus $A\subset\CC/\ZZ$ which sends $S$ to the circle $\TT=\RR/\ZZ$. 
The map
$\tl f:S\to S$ from \S\ref{sec:glue} projects to a $C^k$ multi-critical circle mapping $\Psi:\TT\to\TT$ which is analytic everywhere except for a single point $\hat x\in\TT$. However, it has two analytic continuations $\Psi_1$, $\Psi_2$ to a neighborhood of $\hat x$.
The map $\Psi$ is not uniquely defined, and depends on the choice of the conformal uniformization of $C$. Let us call any such map a {\it cylinder map} of $\zeta$. Note the obvious parallels with the cylinder renormalization construction (see e.g. \cite{GorYam}).

In view of complex {\it a priori} bounds and  Theorem~\ref{thm-comm}, we have:


\begin{proposition}
  \label{prop-comm2}
  For every pair of integers $B\ge 1$ and $D\ge 3$, there exist a real-symmetric annulus $A\Supset\TT$ and  constants $K,\hat\eps>0$, such that the following holds. For any almost commuting pair $\zeta$ with an irrational rotation number of type bounded by $B$ and the total criticality bounded by $D$, there exists $m_0>0$ with the property that for every $m\ge m_0$, the renormalization $\cR^m\zeta$ has a cylinder map $\Psi$ defined in $A$.
  Furthermore, in the above notations, $\Psi_1$, $\Psi_2$ are defined and conformal in the disk $D_{\hat\eps}(\hat x)$ of radius $\hat\eps$ around $\hat x$, and
  \begin{equation}\label{Psi12_diff_eq}
  \|\Psi_1(z)-\Psi_2(z)\|^\infty_{D_{\hat\eps}(\hat x)}< K\|[\cR^m\zeta]\|^\infty_{U(0)},
  \end{equation}
  where $U(0)$ is as in Theorem~\ref{thm-comm}. 
  The constant $m_0$ can be chosen independently from the pair $\zeta$, provided that $\zeta$ belongs to some fixed family $\mathcal C(\tl\eps,\tl M,D)$.
  \end{proposition}
  
\begin{proof}
According to Proposition~\ref{prop-bounds1}, there exists $m_0=m_0(\zeta)\in\bbN$ such that for every $m\ge m_0$, the renormalization $\cR^m\zeta$ is contained in some fixed class $\mathcal C(\eps,M,D)$, where $\eps$ and $M$ depend only on the parameters $B$ and $D$. If $\zeta$ belongs to a fixed family $\mathcal C(\tl\eps,\tl M,D)$, then, according to the same proposition, the constant $m_0$ can be chosen independently from $\zeta$. Let $\zeta_m=(\eta_m,\xi_m)$ denote the $m$-th renormalization $\cR^m\zeta$, and let $\phi_{\zeta,m}$ be a conformal map from $U(0)$ to a neighborhood of the point $\hat x$ on the cylinder $\bbC\slash\bbZ$, obtained as the restriction of a conformal identification of the Riemann surface $C$ with an annulus $A\subset\bbC\slash\bbZ$ in the case of the gluing procedure for the pair $\zeta_m$.  

Due to compactness of the family $\mathcal C(\eps,M,D)$, there exists $\eps_0>0$, such that for any pair $\zeta$ and any iterate $m\ge m_0(\zeta)$, there exists a point $x_0\in \frac{1}{2}U(0)$ with the property that the maps $\eta_m$ and $\eta_m\circ\xi_m$ are conformal in the domain $D_{\eps_0}(x_0)$, the map $\xi_m^{-1}$ is conformal in the domain $\eta_m(\xi_m(D_{\eps_0}(x_0)))$ and 
$$
\|(\xi_m^{-1})'\|_{\eta_m(\xi_m(D_{\eps_0}(x_0)))}^\infty<K_1,
$$
for some constant $K_1>1$. Then 
 the branches $\Psi_1$ and $\Psi_2$ are conformal in the domain $\phi_{\zeta,m}(D_{\eps_0}(x_0))$ and can be represented as
$$
\Psi_1 = \phi_{\zeta,m}\circ\eta_m\circ\phi_{\zeta,m}^{-1}\qquad\text{and}\qquad \Psi_2 = \phi_{\zeta,m}\circ\xi_m^{-1}\circ\eta_m\circ\xi_m\circ\phi_{\zeta,m}^{-1}.
$$
Hence,
$$
\|\phi_{\zeta,m}^{-1}\circ\Psi_1\circ\phi_{\zeta,m}-\phi_{\zeta,m}^{-1}\circ\Psi_2\circ\phi_{\zeta,m}\|^\infty_{D_{\eps_0}(x_0)} = \| \eta_m - \xi_m^{-1}\circ\eta_m\circ\xi_m\|^\infty_{D_{\eps_0}(x_0)} \le
$$
$$
\le\|(\xi_m^{-1})'\|_{\eta_m(\xi_m(D_{\eps_0}(x_0)))}^\infty\cdot \|[\cR^m\zeta]\|_{D_{\eps_0}(x_0)}^\infty +o(\|[\cR^m\zeta]\|_{D_{\eps_0}(x_0)}^\infty)< 
$$
$$
<K_1 \|[\cR^m\zeta]\|_{D_{\eps_0}(x_0)}^\infty + o(\|[\cR^m\zeta]\|_{D_{\eps_0}(x_0)}^\infty).
$$

We may assume that $\eps_0$ is sufficiently small, so that ${D_{\eps_0}(x_0)}\subset U(0)$,  
hence, due to Theorem~\ref{thm-comm} 
we have 
$$
\|\phi_{\zeta,m}^{-1}\circ\Psi_1\circ\phi_{\zeta,m}-\phi_{\zeta,m}^{-1}\circ\Psi_2\circ\phi_{\zeta,m}\|^\infty_{D_{\eps_0}(x_0)} < K_1 \|[\cR^m\zeta]\|^\infty_{U(0)},
$$
after possibly increasing the constant $K_1$.

Finally, due to compactness of the family $\mathcal C(\eps,M,D)$, we may assume that the map $\phi_{\zeta,m}$ has distortion, bounded by some constant that is independent from the pair $\zeta$ and the iterate $m\ge m_0$. Hence, there exists a constant $K>1$, such that
\begin{equation}\label{psi_norm_eq}
\|\Psi_1-\Psi_2\|^\infty_{\phi_{\zeta,m}(D_{\eps_0}(x_0))} < K \|[\cR^m\zeta]\|_{U(0)}^\infty.
\end{equation}
Furthermore, the domain $\phi_{\zeta,m}(D_{\eps_0}(x_0))$ has definite size in $\phi_{\zeta,m}(U(0))$, and there exists $\hat\eps>0$ independent of $\zeta$ and $m\ge m_0$, such that $D_{2\hat\eps}(\hat x)\subset \phi_{\zeta,m}(U(0))$, and both branches $\Psi_1$ and $\Psi_2$ are defined and analytic in $D_{2\hat\eps}(\hat x)$. Then, by the Koebe Theorem for branched coverings of bounded degree, we conclude that possibly after increasing the constant $K$, the inequality~(\ref{psi_norm_eq}) implies~(\ref{Psi12_diff_eq}), which completes the proof.

\end{proof}

Again, due to standard compactness considerations, there exists $\alpha=\alpha(B,D)>0$, such that we can always find two closed sectors $S_+, S_-\subset A$ with the common vertex at $\hat x$ and the size of the angles equal to $\alpha$, such that $S_+$ lies above the circle $\bbT$, $S_-$ is a reflection of $S_+$ about the circle $\bbT$, $S_+\cap S_- = \{\hat x\}$, and the two branches $\Psi_1$ and $\Psi_2$ are conformal in both sectors. It is then straightforward to carry out a smooth interpolation between these branches in both sectors to obtain:

\begin{proposition}
	\label{prop-comm3}
	There exists an annulus $A_1\Supset\TT$, and $K_1>0$ such that the following holds. In the notation of 
	Proposition~\ref{prop-comm2}, the map $\Psi|_\TT$ has a real-symmetric quasiregular extension $\tl\Psi$ defined in $A_1$, whose dilatation is bounded by
	$K_1||[\cR^m\zeta]||^\infty_{U(0)}$. The annulus $A_1$ and the constant $K_1$ depend only on the parameters $B$ and $D$.
\end{proposition}


Finally, we give a proof of Theorem~\ref{thm-near-comm2}.

For a real-symmetric annulus $A\subset\bbC\slash\bbZ$, let $\tl A\subset\bbC$ be the universal cover of $A$, such that $A = \tl A\slash\bbZ$. The uniform metric on the cylinder maps of $A$ is naturally defined as
$$
\dist_{A}(F,G) = \inf_{\tl F,\tl G} \|\tl F - \tl G\|_{\tl A}^\infty,
$$
where the infinum is taken over all lifts $\tl F, \tl G\colon \tl A\to\bbC$ of the cylinder maps $F$ and $G$ respectively. Similarly, for any $r\ge 0$, the $C^r$ distance between the restrictions of two such cylinder maps $F$ and $G$ to the circle is defined as
$$
\dist_{C^r}(F,G) = \inf_{\tl F,\tl G} \|\tl F - \tl G\|_{C^r(\bbR)},
$$
where the infinum is again taken over all lifts $\tl F, \tl G\colon \tl A\to\bbC$. 

\begin{proof}[Proof of Theorem~\ref{thm-near-comm2}]
Given an almost commuting pair $\zeta$ as in the statement of Theorem~\ref{thm-near-comm2}, let $m_0 = m_0(\zeta)$ be the same as in Proposition~\ref{prop-comm2}. Then, for any $m\ge m_0$, consider the corresponding cylinder map $\Psi$ from Proposition~\ref{prop-comm2}, obtained via a uniformization coordinate $\pi\colon C\to A$, and the quasiregular map $\tl\Psi$ from Proposition~\ref{prop-comm3}. 
Applying Measurable Riemann Mapping Theorem to $\tl\Psi$ (i.e., postcomposing with the solution of the corresponding Beltrami equation), we obtain an analytic multicritical circle map $\check\Psi$ which is defined in the cylinder $A_1$. 
Furthermore, due to Proposition~\ref{prop-comm3} and Theorem~\ref{thm-comm}, the map $\check\Psi$ is exponentially in $m$ close to $\Psi$ in the uniform metric $\dist_{A_1}$. 

Next, consider the map $\hat\Psi:= T_\alpha\circ\check\Psi$, where $T_\alpha\colon\bbC\slash\bbZ\to\bbC\slash\bbZ$ is the rigid rotation by $\alpha$, and $\alpha$ is chosen so that $\rho(\hat\Psi) = \rho(\Psi)$. Due to monotonicity property of the rotation numbers, the angle $\alpha\in\bbR$ can be chose so that
$$
|\alpha|\le \dist_{C^0}(\check\Psi,\Psi) \le \dist_{A_1}(\check\Psi,\Psi).
$$
Thus, the maps $\hat\Psi$ and $\Psi$ are still exponentially in $m$ close with respect to the uniform metric $\dist_{A_1}$.

Finally, let us pull back the renormalization $\cR\hat\Psi$ by the uniformizing coordinate $\pi\colon C\to A$ to obtain an analytic commuting pair $\zeta_{m+1}$. According to Lemma~\ref{lem:lip} and Remark~\ref{Lip_remark}, the renormalizations $\cR\Psi$ and $\cR\hat\Psi$ are exponentially close in $m$ with respect to the $C^0$-metric. Hence, it follows that $\zeta_{m+1}$ is exponentially in $m$ close to $\cR^{m+1}\zeta$.

As the last step, we observe that the maps $\Psi$ and $\hat{\Psi}$ have the same number of critical points, and the sum of the orders of these critical points does not exceed $2D$. In particular, due to Koebe Theorem, this implies that the map $\hat{\Psi}$ belongs to some normal family $\tl{\mathcal C}_{B,D}$ of maps of the annulus $A_1$. Furthermore, this family is completely determined by the constants $B$ and $D$, and is compact with respect to the metric $\dist_{C^r}$, for any $r\ge 0$. This implies that $\zeta_{m+1}$ belongs to some compact family $\mathcal C_{B,D}$ that is also determined by the constants $B$ and $D$.


\end{proof}

\subsection{Renormalizations of smooth pairs converge to almost commuting ones}


We state the following:
\begin{theorem}
  \label{analytic-convergence}
  For any positive integers $B\ge 1$, $D\ge 3$, there exist $\lambda_2\in(0,1)$ and a family $\cA_{B,D}=\mathcal C(\tl\eps,\tl M, \tl D)$ of almost commuting pairs such that the following holds. For any $C^3$ multicritical commuting pair $\zeta$ with rotation number of type bounded by $B$ and the total criticality not greater than $D$, there exist $M>0$ and a sequence of almost commuting pairs $\{\zeta_n\}_{n=1}^\infty\subset\cA_{B,D}$, such that for every $k\in\NN$ we have
\begin{equation}\label{smooth_conv_eq}
\dist_{C^0}(\zeta_k, \cR^k\zeta)\le M\lambda_2^k.
\end{equation}
Furthermore, the total criticalities of $\zeta_k$ and $\zeta$ are the same, the number of critical points of $\zeta_k$ is not larger than the number of critical points of $\zeta$, and $\rho(\zeta_k) = \rho(\cR^k\zeta)$.



\end{theorem}
\begin{proof}
The approach in the proof
follows the general outline of the standard Shuffling Lemma in one-dimensional renormalization theory (see \cite[Chapter VI, Theorem 2.3]{deMelo_vanStrien}), so we will be brief. 
Note also
a parallel with Theorem~A.6 of \cite{FM1}, where~(\ref{smooth_conv_eq}) is established for analytic approximations in $C^2$ metric by unicritical pairs, which are almost commuting to the $0$-th order, using a similar approach.

  Set $p\cR^k\zeta=(\eta_k,\xi_k)$. Let us consider the first map of the pair, the construction being identical for the second one. We have
  $$\eta_k=f_{q_{k+1}}\circ f_{q_{k+1}-1}\circ\cdots\circ f_2\circ f_1,$$
    where each of the maps $f_j$ is a restriction of either $\eta$ or $\xi$ to the appropriate element $I_j$ of the $k$-th dynamical partition of $\zeta$.
    We have two possibilities: either $f_j$ is a diffeomorphism, or it is a map with finitely many critical points. Let us assume that $k$ is large enough, so that $f_j$ can have at most one critical point, and the geometry of the level $k$ dynamical partition of $\zeta$ is bounded by the constant $C=C(B,D)$.
    
    We replace each of the diffeomorphisms $f_j$ in the composition by $g_j$, which is the standard osculating polynomial of degree $2D+2$, interpolating between the values of $f_j^{(m)}$, $m=0,\ldots, D$ at the endpoints of $I_j$. 

    In the case when $f_j$ has a critical point, it is a restriction of a composition $\phi\circ Q_m\circ\psi$ as in Definition~\ref{commut_pair_def} \ref{com_pairs_crit_property}. By inserting linear rescalings into the composition, let us rewrite it as
    \begin{equation}
      \label{eq:compfj}
f_j=\hat\phi\circ(x-a_j)^m\circ\hat\psi,
    \end{equation}
    where $\hat\phi$ and $\hat\psi$ are homeomorphisms between closed intervals, and $\hat\psi(I_j)=[-1,1]$.
    We again replace $\hat\psi$ and $\hat\phi$ with the osculating polynomials $p_1$, $p_2$ respectively,
    interpolating between the values of their derivatives up to order $D$ at the endpoints of their intervals of definitions. We let
    \begin{equation}
      \label{eq:compgj}
g_j=p_2\circ(x-a_j)^m\circ p_1.
    \end{equation}

    By selecting a large enough $k$, we can guarantee that all of the osculating polynomials defined above do not have any critical points on neighborhoods                       around their intervals of definition, whose size is some definite proportion of the diameter of the intervals.
    
    We then set
  $$\hat \eta_k=g_{q_{k+1}}\circ g_{q_{k+1}-1}\circ\cdots\circ g_2\circ g_1,$$
  and similarly for $\hat\xi_k$.

  Finally, we let $\tl \zeta_k=(\tl\eta_k,\tl\xi_k)$ stand for the pair $(\hat\eta_k,\hat\xi_k)$ rescaled by the linear change of coordinates
  $$x\mapsto \hat\xi_k(0)\cdot x.$$
  By construction, $\tl\zeta_k$ is an almost commuting pair of the appropriate order.
 Note, that the construction of an approximating pair $\tl\zeta_k$ is like the one carried out in the classical Shuffling Lemma. The only difference is that in the
  latter, osculating polynomials of degree $2$ are used, instead of the higher degree polynomials we use in the present paper, to ensure approximate commutation to order $D$.
 Applying the argument in the Shuffling Lemma {\it mutatis mutandis}, we see that 
  $$d_k\equiv \dist_{C^0}(\tl\zeta_k,\cR^k\zeta)\to 0$$
  at a geometric rate.



  Observe, that there exists $\eps=\eps(B)$ such that, for $n$ large enough, if $\dist_{C^0}(\zeta_1,\cR^m\zeta_0)<\eps$ for a pair $\zeta_1$,
  then $\chi(\zeta_1)=\chi(\cR^m\zeta_0)$ which  puts us in a position to apply the Lipschitz property of renormalization (Lemma~\ref{lem:lip}). Applying it inductively, as well as using straightforward considerations of real {\it a priori} bounds, we see that the first
  $\ell$ digits of the continued fractions of $\tl\zeta_k$ and $\cR^k\zeta_0$ coincide, and moreover, their respective critical points lie in the same intervals of the $\ell$-th dynamical partition for $\ell=O(|\log(d_k)|)$.

  It remains to ``correct'' the rotation number of $\tl\zeta_k$.
  Let $f=f_k$ be a corresponding map of the circle (\ref{eq:fzeta}), obtained by the gluing
$$[x,\tl\xi]\overset{\psi}{\longrightarrow}{\TT}.$$
  Consider the family
  $$g_{b}=f+b\mod\ZZ\text{ for }b\in\TT.$$
  Since the combinatorics of $f$ was already correct to the level $\ell$, there exists $b=O(d_k)$ for which $g_b$ has the desired rotation number.
  We lift this circle map to obtain  $$\eta_k=\phi^{-1}\circ g_{b}\circ \phi,$$
  and and similarly for $\xi_k$, to complete the construction.

\end{proof}


\subsection{Proofs of the main results for $C^3$ maps}
\label{subsec-proofs-smooth}

\begin{proof}[Proof of Theorem~\ref{th:rencov1}]


Throughout the proof we assume that the two $C^3$-smooth bi-cubic maps $f$ and $g$ are such that the forward orbit of any critical point under the dynamics of the map does not contain the other critical point. Otherwise, the renormalizations of the maps are eventually unicritical, and their exponential convergence follows directly from the main result of~\cite{GdM}.

Consider the space $\mathcal C_B$ of all \textit{analytic} bi-cubic multicritical circle maps with irrational rotation numbers of type bounded by $B$ and with topology, defined by the $C^2$ distance. 
Let $\tl{\mathcal I}_B$ be the union of all $\omega$-limit sets of all orbits of renormalization $\{\cR^nf\}_{n=0}^\infty$ with $f\in\mathcal C_B$. By
complex {\it a priori} bounds (Theorem~\ref{ComplexBounds_theorem}), the set $\tl{\mathcal I}_B$ is sequentially precompact in $\mathcal C_B$, hence, the closure of $\tl{\mathcal I}_B$ is compact. We denote this closure by $\mathcal I_B$. We note that by construction, the set $\mathcal I_B$ is forward invariant under renormalization. It also follows from Theorem~\ref{th:rencov2} that the renormalization orbit of any $f\in\mathcal C_B$ accumulates on the set $\mathcal I_B$, so the set $\mathcal I_B$ is called the \textit{attractor of renormalization} of combinatorial type bounded by $B$. 


Let $f$ be any $C^3$-smooth bi-cubic multicritical circle map with an irrational rotation number of type bounded by $B$. Then let $\{\zeta_m\}_{m=0}^\infty$ be the corresponding sequence of almost commuting pairs from Theorem~\ref{analytic-convergence}. In particular, this means that there exists a constant $M_2=M_2(f)>0$, such that 
\begin{equation}\label{conv_1_eq}
\dist_{C^0}(\zeta_k, \cR^k f)\le M_2(f)\lambda_2^k,
\end{equation}
for any integer $k\ge 0$, where $\lambda_2\in (0,1)$ is the same as in Theorem~\ref{analytic-convergence}, and $\rho(\zeta_k)=\rho(\cR^k\zeta)$.

Now, for any integer $m\ge 0$, let $\{\zeta_{m,k}\}_{k=0}^\infty$ be the sequence of analytic commuting pairs as in Theorem~\ref{thm-near-comm2}, corresponding to the almost commuting pair $\zeta_m$. In particular, for any integer $k\ge 0$, we have $\rho(\zeta_{m,k})=\rho(\cR^k\zeta_m) = \rho(\cR^{m+k}f)$, and 
\begin{equation}\label{conv_2_eq}
\dist_{C^0}(\zeta_{m,k}, \cR^k \zeta_m)\le M_1\lambda_1^k,
\end{equation}
where $\lambda_1\in (0,1)$ is the same as in Theorem~\ref{thm-near-comm2}, and the constant $M_1$ depends on the pair $\zeta_m$. Since, according to Theorem~\ref{analytic-convergence}, the pairs $\zeta_m$ can be chosen from some fixed set $\mathcal C(\tl\eps,\tl M,D)$ that depends only on $B$ and the total criticality $D$, it follows that $M_1$ can be assumed to depend only on $B$ and the total criticality, which is equal to $9$ in the bi-cubic case. 

Since according to Theorem~\ref{thm-near-comm2} and Theorem~\ref{analytic-convergence}, the pairs $\zeta_m$ and $\zeta_{m,k}$ are chosen from two $C^2$-compact sets that depend only on the constant $B$, it follows that the dynamical partitions of all pairs $\zeta_m$ and $\zeta_{m,k}$ at all levels, satisfy conditions~(1) and~(2) of Definition~\ref{bound_geom_def} for some constant $C=C(B)>0$. Without loss of generality, we may assume that the constant $C$ is not smaller than the one from Theorem~\ref{real_bounds_thm1}. Then, according to Theorem~\ref{real_bounds_thm1}, for all sufficiently large integers $m$, the geometry of all dynamical partitions of $\cR^m f$ is bounded by $C$.

Let the constant $L=L(C, B, D)>0$ (where $D=9$) be the same as in Lemma~\ref{lem:lip}. Next, combining~(\ref{conv_1_eq}) and~(\ref{conv_2_eq}) with the result of Lemma~\ref{lem:lip}, we get
$$
\dist_{C^0}(\cR^{m+k}f, \zeta_{m,k})\le M_2(f)\lambda_2^m L^k  + M_1\lambda_1^k,
$$
for all non-negative integers $k$ and all sufficiently large integers $m$. Since any integer $n\in\bbN$ can be represented as $n=m+k$, where $m, k, m/k \to\infty$ as $n\to\infty$, it follows that there exist constants $\lambda_3 = \lambda_3(B)\in (0,1)$ and $M_3 = M_3(f)>0$, and a sequence of analytic commuting pairs $\{\hat\zeta_n\}_{n=0}^\infty$ from a fixed $C^2$-compact set, such that $\rho(\hat\zeta_n)=\rho(\cR^nf)$, and 
$$
\dist_{C^0}(\cR^{n}f, \hat\zeta_{n})\le M_3(f)\lambda_3^n,
$$
for all non-negative integers $n$.

Finally, we apply Theorem~\ref{th:rencov2} that establishes exponential convergence of renormalizations of analytic bi-cubic pairs to the attractor $\mathcal I_B$. Following the same strategy as in the previous paragraph, we conclude the existence of a sequence of analytic pairs $\{\tl\zeta_n\}_{n=0}^\infty\subset\mathcal I_B$ and the numbers $\lambda_4 = \lambda_4(B)\in (0,1)$ and $M_4 = M_4(f)>0$, such that
\begin{equation}\label{conv_3_eq}
\dist_{C^0}(\cR^{n}f, \tl\zeta_{n})\le M_4(f)\lambda_4^n,
\end{equation}
for all non-negative integers $n$.

According to \cite{EsYam2021}, the attractor $\mathcal I_B$ is hyperbolic in the appropriate space of analytic bi-cubic maps. Hence, according to the Shadowing Lemma, there exist $K=K(B)>0$ and $\delta_0 = \delta_0(B)>0$, such that for any $0<\delta<\delta_0$, any $\delta$-pseudo orbit in $\mathcal I_B$ is $(K\delta)$-shadowed by a true orbit in $\mathcal I_B$. 
Inequality~(\ref{conv_3_eq}) combined with Lemma~\ref{lem:lip}, imply existence of a sufficiently large $n_0=n_0(B, f)$, such that
\begin{equation*}
\dist_{C^0}(\cR\tl\zeta_n, \tl\zeta_{n+1})\le (L+\lambda_4)M_4(f)\lambda_4^n<\delta_0,
\end{equation*}
for all $n\ge n_0$. Set $M_5(f) = (K L+K\lambda_4+1)M_4(f)$. Then, according to the Shadowing Lemma and the triangle inequality, for any 
$m\ge n_0$, there exists a commuting pair $f_m\in\mathcal I_B$, such that the inequality
\begin{equation}\label{conv_4_eq}
\dist_{C^0}(\cR^{n+m} f, \cR^nf_m) \le M_5(f)\lambda_4^{m}
\end{equation}
holds for any integer $n\ge 0$. In what follows, we may always assume that $m$ is sufficiently large, so that the right hand side in~(\ref{conv_4_eq}) is always smaller than the constant $\varepsilon$ from Lemma~\ref{combinatorics_vs_distance_lemma}. Hence, according to Lemma~\ref{combinatorics_vs_distance_lemma}, the commuting pairs $\cR^mf$ and $f_m$ have the same signature (i.e., $\sigma(\cR^mf) = \sigma(f_m)$, for all sufficiently large $m$).

Finally, assume that $g$ is another $C^3$-smooth bi-cubic circle map with $\sigma(g)=\sigma(f)$. Then, for any sufficiently large integer $m>0$ and any $n\ge 0$, we have 
\begin{equation}\label{conv_5_eq}
\dist_{C^0}(\cR^{n+m} g, \cR^ng_m) \le M_5(g)\lambda_4^{m},
\end{equation}
where $g_m\in\mathcal I_B$ is some analytic bi-cubic pair with $\sigma(g_m)=\sigma(f_m)$. Since the signatures of $f_m$ and $g_m$ match, it follows from Theorem~\ref{th:rencov2} and $C^2$-compactness of the attractor $\mathcal I_B$ that
\begin{equation}\label{conv_6_eq}
\dist_{C^0}(\cR^nf_m,\cR^ng_m) \le K_1\lambda^n,
\end{equation}
for some constants $K_1>0$, $\lambda\in(0,1)$ that depend only on $B$. Combining inequalities~(\ref{conv_4_eq}), (\ref{conv_5_eq}) and (\ref{conv_6_eq}) via the Triangle Inequality and setting $m\sim n$, we obtain exponential convergence
$$
\dist_{C^0}(\cR^kf,\cR^kg) \to 0,\qquad\text{as }k\to\infty,
$$
Which completes the proof of Theorem~\ref{th:rencov1}.
\end{proof}

Finally, Theorem~\ref{main_rigidity_theorem} follows from Theorem~\ref{th:rencov2}, Theorem~\ref{th:rencov1} and 
Proposition~\ref{dFdM43_prop}.